\documentclass[12pt]{amsart}
\usepackage{amsmath,amsfonts,amssymb,amscd,amsthm,amsbsy,epsf,graphicx}   
\numberwithin{equation}{section}
\newtheorem{thm}{Theorem} 
\newtheorem{lem}{Lemma} 
\newtheorem{cor}[thm]{Corollary}
\newtheorem{prop}[thm]{Proposition}  
\theoremstyle{definition}

\def\bye{\end{document}} \def\by{\end{proof}\end{document}}

\def\R{{\mathbb R}}
\def\N{{\mathbb N}}

\def\diam{\operatorname{diam}}  

\def\tr{\operatorname{tr}}

\newcommand{\overbar}[1]{\mkern 1.9mu\overline{\mkern-1.9mu#1\mkern-0.1mu}
\mkern 0.1mu}

\def\mid{\,:\,}
\def\dist{\mathrm{dist}}
\def\supp{\mathrm{supp}\,}
\def\USC{\operatorname{USC}}
\def\LSC{\operatorname{LSC}}
\DeclareMathOperator{\argmin}{arg\, min}
\DeclareMathOperator{\EXP}{\mathbb{E}}
\DeclareMathOperator{\PROB}{\mathbb{P}}
\DeclareMathOperator{\e}{\mathrm{e}}
\DeclareMathOperator{\di}{\mathrm{d}}
\def\1{\mathbf{1}}

\def\cS{\mathcal{S}}
\def\Lip{\operatorname{Lip}}
\def\L{\mathcal{L}}

\author{Hitoshi Ishii$^{1,*}$ and Panagiotis E. Souganidis${}^{2}$}
\thanks{${}^*$ Corresponding author}
\thanks{${}^{1}$ Faculty of Education and Integrated Arts and Sciences, Waseda University, Nishi-Waseda, Shinjuku, Tokyo 169-8050, Japan / Faculty of Science, King Abdulaziz University, P. O. Box 80203 Jeddah, 21589 Saudi Arabia. 
Partially supported by the KAKENHI \#21224001,
\#23340028 and \#23244015, JSPS} 
\thanks{${}^{2}$ Department of Mathematics, The University of Chicago, 5734 S. University Avenue, Chicago, IL 60657, USA. 
Partially supported by the National Science Foundation grants DMS-0901802 and DMS-1266383}
\keywords{parabolic equation, asymptotic behavior, metastability, stochastic perturbation}
\subjclass[2010]{Primary 35B40; Secondary 35K20, 37H99}
\email{hitoshi.ishii@waseda.jp (Hitoshi Ishii),
souganidis@math.uchicago.edu (Panagiotis E. Souganidis)}
\date{\today}

\def\bbS{\mathbb{S}}  
\title{Metastability for parabolic equations 
with drift: part 1}

\begin{document}
\maketitle
\begin{abstract}We provide a self-contained analysis, based entirely on pde methods, of the
exponentially long time behavior of solutions to linear uniformly parabolic
equations which are small perturbations of
a transport equation with vector field having a globally stable point. We show that the solutions converge to a constant, which is either the initial value at the stable point or the boundary value at the minimum of the associated
quasi-potential. This work extends previous results of Freidlin and
Wentzell and Freidlin and Koralov and applies also to semilinear elliptic
pde. 
\end{abstract}

\tableofcontents    

\section{Introduction}\label{sec:int}

\noindent In this paper we provide a self-contained  analysis, based on entirely pde methods, of the long time behavior (at scale $\exp \lambda/\varepsilon)$, as $\varepsilon\to 0$, of the solution $u^\varepsilon=u^\varepsilon(x,t)$ of the the parabolic equation 
\begin{equation}\label{eq:int1}
u_t^\varepsilon=L_\varepsilon u^\varepsilon\ \ \ \text{ in } \ Q:=\varOmega\times(0,\,\infty),  
\end{equation}
where, for $\phi$ smooth,   the elliptic operator $L_\varepsilon$ is given by 
$$
L_\varepsilon\phi(x)
:=\varepsilon\tr[a(x)D^2\phi(x)]+b(x)\cdot D\phi(x)
$$
and the initial-boundary condition
\begin{equation}\label{eq:boundary}
u^\varepsilon= g \  \text{ on } \  \partial_{\mathrm{p}} Q:= (\overbar \varOmega \times \{0\}) \cup (\partial \varOmega \times (0,\infty)).
\end{equation}
 
\noindent Here $\varepsilon>0$,  $\Omega \subset \R^n$, $a(x)=(a_{ij}(x))_{1\leq i,j\leq n}\in \bbS^n$, the space of $n\times n$ symmetric matrices, is positive, 
``$\tr$'' and ``$\cdot$'' denote the trace of square matrices 
and the inner product in Euclidean spaces respectively  and the vector field $b$ has some $x_0 \in \varOmega$ as an asymptotically stable equilibrium. Exact assumptions are stated below. 
\smallskip

\noindent Roughly speaking the result states that there exists $m_0>0$ and some $x^*\in \partial \varOmega$ such that, as $\varepsilon\to 0$ and locally uniformly in $\varOmega$, 
$$ u^\varepsilon(x, \lambda/\varepsilon) \to g(x_0)  \ \text{ if } \ \lambda < m_0 \  \text{ and }  \   u^\varepsilon(x, \lambda/\varepsilon) \to g(x^*)  \ \text{ if } \ \lambda > m_0.$$
 
\noindent Our work extends previous results of Freidlin and Wentzell \cite[Chap. 4]{FW2012} (see also Freidlin and Koralov \cite{FK2010,FK2012}) who studied, using probabilistic 
techniques, the asymptotic behavior of the $u^\varepsilon$'s  for $a=I$, the identity matrix of order $n$. 
\smallskip

\noindent To 
make precise statements as well as to provide an interpretation of the results in terms of the metastability properties of random perturbations of some ordinary differential equations (ode for short), 
we introduce next the assumptions (A$1$)--(A$5$) which will hold throughout. In what follows, 
$B_r(x) $ is the open ball in $\R^n$ centered at the $x$ with radius $r$ and $B_r:=B_r(0)$. Moreover to simplify the notation throughout the paper we take $x_0=0$.

\begin{itemize} 
\item[(A1)] (Regularity) The symmetric matrix  $a(x)$ and the vector field $b(x)$ are Lipschitz continuous in $\R^n$. 
\item[(A2)](Uniform ellipticity) There exists a constant $\theta\in (0,\,1)$ such that
\[
\theta I\leq a(x)\leq \theta^{-1} I \  \text{ for all } x\in\R^n
\]
\item[(A3)] The set $\varOmega$ is a bounded, open, connected subset of $\R^n$ with $C^1$-boundary.
\end{itemize}
\smallskip

\noindent We consider the dynamical system generated by the ode
\begin{equation}\label{ds}
\dot X=b(X),
\end{equation}
where $\dot X$ denotes the derivative of the function $t\to X(t)$. The solution of  \eqref{ds} 
with  initial condition $X(0)=x\in\R^n$ is denoted by $X(t;x)$. 
The assumptions on $b$ are: 

\begin{itemize} 
\item[(A4)](Global asymptotic stability)\\
(i)  For any $x\in\R^n$, \ ${\displaystyle \lim_{t\to \infty}X(t;x)=0}$.\\
(ii)  For any $\delta>0$ there exists $r>0$  such 
that,  for all $x\in B_r$ and  $t\geq 0$,  
\ $X(t;x)\in B_\delta$.\\
\item[(A5)] $b(x)\cdot\nu(x) <0$ \ on  \ $\partial\varOmega$, 
where $\nu(x)$ is the 
exterior unit normal at $\,x\in\partial\varOmega$. 
\end{itemize}
\smallskip

\noindent We remark that (A4) implies that 
$$b(0)=0 \ \text{ and } \ b\not=0 \ \text{in} \  
\R^n\setminus\{0\},$$ 
and that (A5) ensures  
that  $\varOmega$ (resp.  $\overbar\varOmega$) 
is positively invariant under the flow\\ $X:\R\times\R^n\to\R^n$, that is, for all $(x,t)\in\varOmega\times[0,\,\infty)$ 
(resp.  $(x,t)\in\overbar\varOmega\times[0,\,\infty)$),
$$X(t;x)\in\varOmega \ \  (\text{resp. } X(t;x)\in\overbar\varOmega).$$ 

\begin{center}
\includegraphics[height=5cm]{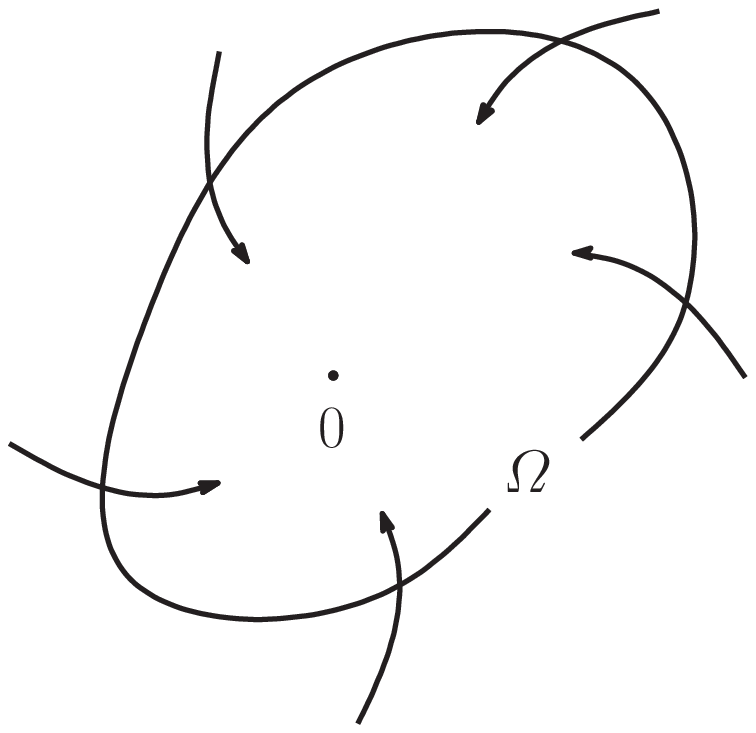}
\end{center}




\noindent The asymptotic behavior of the $u^\varepsilon$'s, as $\varepsilon\to 0$,   is closely related to 
the behavior, as $\varepsilon\to 0$, of the following random perturbation of \eqref{ds} 
\begin{equation}\label{eq:rp1}
\di X_t^\varepsilon=b(X_t^\varepsilon)\di t+\sqrt{2\varepsilon}\sigma (x)\di W_t,
\end{equation}
where $(W_t)_{t \in \R}$ is a standard $n$-dimensional Brownian motion and the matrix $\sigma$ is the square root of $a$, that is, $a=\sigma\sigma^t$, 
and is, in view of (A1) and (A2), also Lipschitz continuous in $\R^n$. In what follows  we write $X_t^\varepsilon(x)$ for the solution of \eqref{eq:rp1} with initial condition $x$. 
\smallskip

\let\gep=\varepsilon
\noindent For each $x\in\varOmega$, let $\tau^\varepsilon_x$ be 
the first exit time of $X_t^\varepsilon(x)$ from $\varOmega$, that is 
\[
\tau^\gep_x:=\inf\{t\geq 0\mid X_t^\gep\not\in\varOmega\}. 
\]
Consider the Hamilton-Jacobi equation
\begin{equation}\label{hj1}
H(x,Du)=0 \ \ \ \text{ in }\ \varOmega,
\end{equation}
where, for $x,p \in \R^n$,
\begin{equation}\label{hj2}
H(x,p)=a(x)p\cdot p+b(x)\cdot p,  
\end{equation}
let $V\in C(\overbar\varOmega)$ be the maximal sub-solution of \eqref{hj1} satisfying 
\ $V(0)=0$, and set \ $m_0=\min_{\partial\varOmega}V$ ---throughout the paper when we refer to solutions of 
Hamilton-Jacobi and ``viscous'' Hamilton-Jacobi equations we mean viscosity solutions. 
\smallskip

\noindent For $a=I$ 
the results of \cite[Chap. 4]{FW2012} state roughly that,  
in average, for any $x\in\varOmega$ and as $\varepsilon\to 0$, 
\begin{equation}\label{eq:rp2}
\tau^\gep_x \approx \e^{m_0/\gep}  \ \text{ and } 
X_t^\varepsilon(x)  \text{ exits from } \ \varOmega \ \text{ near } \ \argmin(V|\partial\varOmega),
\end{equation}
where  
$\argmin(V|\partial\varOmega)$ is the subset of $\partial\varOmega$ where $V$ attains its minimum over 
$\partial\varOmega$. 
\smallskip

\noindent A simple  example that gives an  idea for what is happening is to take 
$a=I$,  $b(x)=-x$ and $\varOmega=B_1$. In this case 
$u(x)=|x|^2/2$ obviously satisfies $H(x,Du(x))=|x|^2-x\cdot x=0$ and, hence, 
$V(x)=|x|^2/2$.  It also follows from elementary stochastic calculus considerations  that, for every $x\in B_1$,  $X^\varepsilon_t(x)= x\exp( {-t}) + \sqrt{2\varepsilon} \int_0^t \exp(s-t) dW_s$. 
\smallskip


\noindent Given $H$ as in \eqref{hj2},  let $L\in C(\R^{2n})$ be its convex conjugate, that is 
$$L(x,\xi):=\frac 14 a(x)^{-1}(\xi-b(x))\cdot (\xi-b(x)),$$
where $a(x)^{-1}$ denotes the inverse matrix of $a(x)$. 
\smallskip

\noindent Following Freidlin-Wentzell \cite{FW2012},  we introduce the quasi-potential $V_\varOmega$ on $\overbar\varOmega\times\overbar\varOmega$ 
\begin{equation}\label{eq:mt1}
\begin{aligned}
V_\varOmega(x,y):=\inf\Big\{\int_0^T L(X(t),\dot X(t))\di t: \ &T>0,\ X\in \Lip([0,\,T],\,\overbar\varOmega), 
\\ 
&X(0)=x,\ X(T)=y \Big\},
\end{aligned}  
\end{equation}
where $\Lip([0,\,T],\overbar\varOmega)$ is the set of all Lipschitz continuous functions $t \mapsto  
X(t)$
such that $X(t)\in\overbar\varOmega$ for all $t\in[0,\,T]$. 
\smallskip

\noindent Next we define the function $V$ and the constant $m_0$ by
$$ V(y):=V_\varOmega(0,y) \  \  \text{and } \ \  m_0=\min_{\partial\varOmega}V.$$ 


\smallskip

\noindent Our main theorem is as follows.

\begin{thm} \label{thm:mt1}
Assume (A1)-(A5) and $g\in C(\overbar\varOmega)$ . For each $\varepsilon>0$, let $u^\varepsilon\in C(\overbar Q)\cap C^{2,1}(Q)$ be the solution of 
\eqref{eq:int1}, \eqref{eq:boundary}. 
\begin{enumerate}
\item[(i)] Fix $\lambda\in(0,\,m_0)$. For any compact subset $K$ of  $\varOmega$ and $\sigma (\varepsilon) >0$ 
such that  $\sigma(\varepsilon) \leq \exp {\lambda/\varepsilon}$ and 
$\lim_{\varepsilon\to 0+} \sigma(\varepsilon)=\infty$, 
\begin{equation}
\lim_{\varepsilon\to 0+}u^\varepsilon(\cdot ,t)=g(0)\quad \text{ uniformly on }\  K\times[\sigma(\varepsilon),\,\e^{\lambda/\varepsilon}].   
\label{eq:mt3}
\end{equation}
\item[(ii)] Assume that $g=g(0)$ on $\argmin(V|\partial\varOmega)$.
For any  compact subset  $K$ of $\varOmega\cup\argmin(V|\partial\varOmega)$ and $\sigma (\varepsilon) >0$ 
such that 
$\lim_{\varepsilon\to 0+} \sigma(\varepsilon)=\infty$,
\begin{equation}
\lim_{\varepsilon\to 0+}u^\varepsilon(\cdot,t)=g(0)\quad \text{ uniformly on } K\times[\sigma(\varepsilon), \infty). 
\label{eq:mt4}
\end{equation}
\item[(iii)]  Fix  $\lambda\in(m_0,\,\infty)$ and assume that $g=g_0$ on $\argmin(V|\partial\varOmega)$ for some constant 
$g_0$.  Then, for every  compact subset $K$ of $\varOmega\cup\argmin(V|\partial\varOmega)$, 
\begin{equation}\label{eq:mt5}
\lim_{\varepsilon\to 0+}u^\varepsilon(\cdot,t)=g_0\quad\text{ uniformly on } \ K\times[\e^{\lambda/\varepsilon},\,\infty). 
\end{equation}
\end{enumerate}
\end{thm}

\noindent Next we use the assertions of the theorem to make precise the statement of \eqref{eq:rp2} for the general random perturbation \eqref{eq:rp1}.  
The solution $u^\varepsilon$ of \eqref{eq:int1}, 
\eqref{eq:boundary} is 
given by  
\[
u^\varepsilon(x,t)=\EXP_x g(\min (t,\tau^\varepsilon_x)),
\] 
where $\EXP_x$ denotes the expectation conditioned on $X_0^\varepsilon=x$.
\smallskip

\noindent For any closed subset $\varGamma$ of $\partial\varOmega$,  let  $g:=\1_\varGamma$ be its characteristic function, 
which, of course,  is not continuous on $\varGamma$ unless $\varGamma=\emptyset$. For this choice of $g$, 
\[
u^\varepsilon(x,t)=\PROB_x(\tau^\varepsilon_x<t,\, X_{\tau^\varepsilon_x}^\varepsilon\in\varGamma),
\] 
where $\PROB_x$ denotes the probability conditioned on $X_0^\varepsilon=x$. 
Note that the  argument here is heuristic 
in the sense that we use  \eqref{eq:mt3} and \eqref{eq:mt5} 
even  for a discontinuous $g$.
\smallskip

\noindent Assume that $\argmin(V|\partial\varOmega)\subset \varGamma$, which 
implies that $g(0)=0$ and $g(x)=1$ for all $x\in\argmin(V|\partial\varOmega)$.   
Then \eqref{eq:mt3} and \eqref{eq:mt5} roughly say that, 
for any $\delta>0$ and any compact $K\subset\varOmega$, as $\varepsilon\to 0$, 
\begin{equation}\label{eq:mt6}
\begin{cases}
\PROB_x\left(\tau^\varepsilon_x<t,\, X_{\tau^\varepsilon_x}^\varepsilon\in\varGamma\right) \ \to \ 0 
&\text{ uniformly on } \ K\times[\e^{\delta/\varepsilon},\,
\e^{(m_0-\delta)/\varepsilon}], \\[3pt] 
\PROB_x\left(\tau^\varepsilon_x<t,\, X_{\tau^\varepsilon_x}^\varepsilon\in\varGamma\right) \ \to \ 1 &\text{ uniformly on } \ K\times
[\e^{(m_0+\delta)/\varepsilon},\,\infty).
\end{cases}
\end{equation}

\noindent It follows from \eqref{eq:mt6}, with $\varGamma=\partial\varOmega$, that for any $\delta>0$ and any compact $K\subset\varOmega$, as $\varepsilon\to 0$, 
\[
\PROB_x\left(\e^{(m_0-\delta)/\varepsilon}<\tau^\varepsilon_x<\e^{(m_0+\delta)/\varepsilon}
\right) \ \to \ 1 \ \ \ \text{ uniformly  on } \  K.  
\]
Finally, this last observation and \eqref{eq:mt6}, with $\varGamma=\argmin(V|\partial\varOmega)$, 
yield that,  for any $\delta>0$ and any compact $K\subset\varOmega$, as $\varepsilon\to 0$,  
\[
\PROB_x\left(\e^{(m_0-\delta)/\varepsilon}<\tau^\varepsilon_x<\e^{(m_0+\delta)/\varepsilon},\, 
X_{\tau^\varepsilon_x}^\varepsilon\in \argmin(V|\partial\varOmega)\right)
\ \to \  1 \ \ \ \text{ uniformly on } \ K.
\]
This is a crude probabilistic interpretation of Theorem \ref{thm:mt1}, which may be used as 
 a justification of \eqref{eq:rp2}.   
\smallskip

\noindent We continue with a brief discussion of the main steps of the proof of Theorem \ref{thm:mt1} as well as an outline of the paper.  
\smallskip

\noindent It turns out,  
and this is explained in Section \ref{sec:qp}, that,
if $u(x):=V_\varOmega(x,y)$ and $v(y):=V_\varOmega(x,y)$, then 
$H(x,-Du)=0 \ \text{ in }\ \varOmega\setminus\{y\} 
 \ \text{ and } \ 
H(y,Dv)=0 \  \text{ in } \ \varOmega\setminus\{x\}.$
\smallskip

\noindent The second property is used in Section \ref{sec:qp} to find, for each $r>0$, a smooth approximation $W_r\in C^2(\overbar \varOmega)$ of $V=V_\varOmega(0,\cdot)$ 
such that   $H(x,DW_r)\leq -\eta \ \text{ in } \ \varOmega\setminus B_r$ 
for some $\eta>0$.
This approximation is used in Section \ref{sec:sts} to analyze the asymptotic behavior of  the $u^\varepsilon$'s in a ``smaller time scale'', that is to prove in Theorem \ref{thm:sts1} that, if  $\lambda>0$ is such that 
$\{V\leq \lambda\}\subset\{g\leq 0\}$, 
then, for each $\delta>0$, there exists $r>0$ such that $\lim_{\varepsilon\to 0+} (u^\varepsilon-\delta)_+=0 \ \text{ uniformly on } \ B_r\times [0,\e^{\lambda/\varepsilon}].$
\smallskip

\noindent The first property of the quasi-potential is used in Section \ref{sec:lts} to find (Proposition  \ref{prop:lts1}), for 
$\lambda>m_0$,  a semiconcave function 
$W\in \Lip(\overbar\varOmega)$  
such that  $ 0<\min_{\overbar\varOmega}W\leq \max_{\overbar\varOmega}W<\lambda$ and 
$H(x,-DW)\geq \eta\,$ 
for some $\eta>0$.
The existence of such $W$ 
allows us to study the asymptotics of the $u^\varepsilon$'s at a ``larger time scale'', that is to prove in Theorem \ref{thm:lts1} that,  for fixed $\lambda>m_0$, any solution $u^\varepsilon\in C(\overbar Q)\cap C^{2,1}(Q)$ of 
\eqref{100} below satisfies $\,
\lim_{\varepsilon\to 0+}u^\varepsilon=0\,$ 
uniformly on $\, \overbar\varOmega\times[\e^{\lambda/\varepsilon},\,\infty).$ To use this result in the proof of the main theorem, we analyze in Section  \ref{sec:sp}
the stationary version of \eqref{eq:int1}, that is the boundary value problem \eqref{eq:sp1}.
\smallskip

\noindent  In Section \ref{sec:ac} we extend  the convergence result  of Theorem \ref{thm:sts1} from convergence on $B_r \times [0,\exp{\lambda/\varepsilon}]$ to uniform convergence on $K \times [T, \exp{\lambda/\varepsilon} - \tau_0]$, for some large $T(r)>0$ and any compact subset $K$ of $\varOmega$ and $\tau_0 >0$. This  ``asymptotic constancy'' is based on the rigorous justification of the fact that the limit, as $\varepsilon\to 0$, of \eqref{eq:int1} is the transport equation $u_t=b\cdot D_xu$.  \smallskip

\noindent The proof of the main theorem is the topic of Section \ref{sec:pmt}. In Section \ref{sec:spe}  we present a generalization of Theorem \ref{thm:mt1} 
to a class to semilinear parabolic equations. Finally, in the Appendix we present a new existence and uniqueness result of viscosity solutions for the class of the semilinear equations considered in Section \ref{sec:spe}. 
\smallskip

\noindent
\emph{Notation and terminology.} We write 
$\overbar  B_R(y,s)$  for the closure of $B_R(y,s)$. We denote by $a\vee b$ and $a\wedge b$ the larger and smaller of $a,\,b\in\R$
respectively and, for $a\in\R$,  $a_+:=a\vee 0$ and $a_-:=(-a)_+$.
For $A\subset\R^m$ and $B\subset\R^k$,  $\Lip(A,B)$ denotes the set of all Lipschitz continuous functions $f:A\to B$ and $\Lip(A)=\Lip(A,\R)$.
For any function $f:A\to B$ we write $\|f\|_{\infty,A}$ for 
$\sup_{x\in A}|f(x)|$ and, if $B=\R$, $\{f<\alpha\}$ (resp. $\{f\leq \alpha\}$) 
for $\{x: f(x)<\alpha\}$ (resp. $\{x: f(x)\leq \alpha\}$). 
Let $f: A\to B$, and let $\{f_\varepsilon\}_{\varepsilon>0}$ 
and $\{K_\varepsilon\}_{\varepsilon>0}$ be collections of functions $f_\varepsilon: A\to B$ and of subsets $K_\varepsilon\subset A$. 
We say that 
$\lim_{\varepsilon\to 0+} f_\varepsilon=f \  \text{ uniformly on }K_\varepsilon$, 
if \ $\lim_{\varepsilon\to 0+}\|f_\varepsilon-f\|_{\infty,K_\varepsilon}=0$. 
\smallskip

\noindent Throughout the paper sub- and super-solutions should be taken to be in the Crandall-Lions viscosity sense. 
In this direction, given $S\subset\R^n$, and  $u:S \to \R^n$ and 
$F, G: S\times \R\times\R^n\times \bbS^n$
we say that the inequality 
\[F(x,u,Du,D^2u)\leq G(x,u,Du,D^2u) \ \ \ \text{ in } \ S
\]
holds in the (viscosity) sub-solution (resp. super-solution) sense if 
we have 
\[
F(x,u(x),D\phi(x),D^2\phi(x))\leq G(x,u(x),D\phi(x),D^2\phi(x))
\] 
for all $(x,\phi)\in S\times C^2(S)$ such that $u-\phi$ takes a maximum (resp. minimum ) 
at $x$. We also use the term ``in the (viscosity) sub-solution sense''
or ``in the (viscosity) super-solution sense'' for strict inequalities, reversed inequalities and sequences of inequalities.

\section{The Quasi-potential and a smooth approximation}\label{sec:qp}
\noindent Here we recall some classical facts about the quasi-potential and then we construct a smooth approximation, which
plays an important role in the rest of the analysis.
\smallskip

\noindent It is well-known from the theory of viscosity solutions (\cite{Li1982, Ba1994, BC1997, 
CIL}) as well the  weak KAM theory (\cite{Fa, FS2005}, 
\cite[Prop. 7.2]{Mi2008})  that 
$V_\varOmega\mid \overbar\varOmega\times\overbar\varOmega\to\R$ is given by
\begin{equation}\label{eq:qp1}
V_\varOmega(x,y)=\sup\{\psi(x)-\psi(y):  \psi\in\cS^-(\varOmega)\},
\end{equation}
where $\cS^-(\varOmega)$ denotes the set of all sub-solutions $\psi\in C(\overbar\varOmega)$ of 
\[
H(x,-D\psi)=0 \ \ \ \text{ in }\ \varOmega;
\]
note that the coercivity of the Hamiltonian implies that $\cS^-(\varOmega)\subset\Lip(\overbar\varOmega)$. 
\smallskip

\noindent Let $u(x):=V_\varOmega(x,y)$ and $v(y):=V_\varOmega(x,y)$. It is immediate that 
$u \in \cS^-(\varOmega)$, that is,  $u$ is a sub-solution of \ $H(x,-Du(x))\leq 0$ \ in $\varOmega$, 
and, for any $\psi\in\Lip(\overbar\varOmega)$,  
\[
\psi\in\cS^-(\varOmega) \ \ \ \text{if and only if } \ \ \ H(x,-D\psi)\leq 0 \ \ \text{ a.e..}
\] 
Moreover, the function $v(y):=V_\varOmega(x,y)$ is a sub-solution of 
\[
H(y,Dv(y))\leq 0 \ \ \ \text{ in } \ \varOmega.
\]
Finally,
\[
H(x,-Du)=0 \ \text{ in }\ \varOmega\setminus\{y\} 
 \ \text{ and } \ 
H(y,Dv)=0 \  \text{ in } \ \varOmega\setminus\{x\}.
\]

\noindent It is obvious from \eqref{eq:mt1} that, for all \  $x,y\in\overbar\varOmega$ \  and  \  $t\geq 0$,  \ $V_\varOmega(x,y)\geq 0$ \ 
and  \  $V_\varOmega(x,X(t;x))=0$. 
Moreover, letting  $t\to\infty$,  it follows that,  for all \ $x\in\overbar\varOmega$, \ $V_\varOmega(x,0)=0$.  It is also easily seen from the definition of $V_\varOmega$ that, for all  $ x,y,z\in\overbar\varOmega$,
\[
V_\varOmega(x,y)\leq V_\varOmega(x,z)+V_\varOmega(z,y).
\]

\noindent Next we state a technical fact that we need for the construction of the above mentioned auxiliary function.

\begin{prop}\label{prop:qp1}There exists $\psi\in C(\overbar\varOmega\setminus\{0\})$ 
such that, for all $r>0$,
\vskip.075in
\quad  $\psi\in\Lip(\overbar\varOmega\setminus B_r), \  b\cdot D\psi= -1$ \ a.e. in  $\varOmega\setminus\{0\}$, \  and \  $\lim_{x\to 0}\psi(x)=-\infty.$
\end{prop}


\noindent Before proving the proposition, we show in the next lemma a localization property of the flow $X(t;x)$. 

\begin{lem}\label{lem:qp1} For any $0<r<R$, there exists  $T=T(r,R)>0$ such that, for all $x\in B_R$ and  $ t\geq T$,
$X(t;x)\in B_r.$ 
\end{lem}

\begin{proof}  In view of the  asymptotic stability of the origin, there exists  $\delta>0$ such that, for all  $x\in B_\delta$ and  $t\geq 0$, 
\ $X(t;x)\in B_r$, 
while  the global asymptotic stability gives that, for each $x\in \overbar B_R$, there exists $t_x>0$ such that 
$X(t_x;x)\in B_\delta$. Moreover the continuous dependence on the initial data of the solutions of \eqref{ds}  
implies that  $X(t_x;y)\in B_\delta$ \ for all $y$ in a neighborhood 
of $x$. Finally, using the compactness of $\overbar B_R$, we find some $T>0$ such that, for each 
$x\in\overbar B_R$ there exists $\bar t_x\in[0,\,T]$ such that 
$X(\bar t_x;x)\in B_\delta. $
This  implies that, for all $t\geq \bar t_x$,  $X(t;x)\in B_r$, 
and, hence,
 $X(t;x)\in B_r$ \ for all $x\in\overbar B_R$ and $t\geq T$.  
\end{proof} 

\noindent We continue with the 

\begin{proof}[Proof of Proposition \ref{prop:qp1}] Fix $R>0$ so that $\overbar\varOmega\subset B_R$,
select $f\in \Lip(\R^n)$ such that 
$f \geq 0, \ f=1 \ \text{on } \ \overbar \varOmega \ \text{and } \ f=0 \ \text{in } \ \R^n \setminus \overbar B_R,$
and consider the transport equation
\[
b\cdot D\psi=-f \ \ \ \text{ in }\ B_R\setminus\{0\}.
\]

\noindent Lemma \ref{lem:qp1} yields 
that, for each $r\in(0,\,R)$,  there exists $T(r,R)>0$ such that, 
for all $x\in \R^n\setminus B_r$  and $t\geq T(r,R)$,
\[
X(-t;x)\in \R^n\setminus B_R. 
\]  

\noindent Next define $\psi: B_R\setminus\{0\}\to\R$ by
\[
\psi(x):=-\int_0^\infty f(X(-t;x))\di t,
\]
and note that, if $r$ and $T(r,R)$ are as above, then, for all  $x\in B_R\setminus B_r$, 
\[
\psi(x)=-\int_0^{T(r,R)}f(X(-t;x))\di t. 
\]
It follows that $\psi$ is Lipschitz continuous on any 
compact subset of $B_R\setminus\{0\}$ and 
\[
b\cdot D\psi=-f \ \ \ \text{ a.e. in }\ B_R\setminus\{0\}. 
\]
Since $b(0)=0$ and $b\in\Lip(\R^n)$, there exists $L>0$ such that $|b(x)|\leq L|x|$ 
for all $x\in\R^n$. This implies that, for all $t\geq 0$,  
\[
|x|=|X(t;X(-t;x))|\leq |X(-t;x)|\e^{Lt}, 
\] 
and, hence, $\lim_{x\to 0}\psi(x)=-\infty.$
\end{proof}

\noindent We continue with some technical consequences  of Proposition \ref{prop:qp1} which are used later in the paper.

\begin{cor} \label{cor:qp1} For each $r>0$ there exist
$\psi_r\in\Lip(\overbar\varOmega)$ and $\eta>0$ such that
\vskip.075in
\quad $H(x,D\psi_r(x))\leq -\eta  \ \text{ a.e. in } \ \varOmega\setminus B_r \ \text{ and } \ 
H(x,D\psi_r(x))\leq 0 \ \text{a.e. in } \ B_r.$

\end{cor}

\begin{proof} Let $\psi\in C(\overbar\varOmega\setminus\{0\})$ be the function constructed in Proposition \ref{prop:qp1}. 
Fix $r>0$ and select $R>0$ so that 
\[
\min_{\overbar\varOmega\setminus B_r}\psi>-R,  
\]
define $\chi_r\in\Lip(\overbar\varOmega)$ by  
\[
\chi_r(x):=
\begin{cases} 
-R &\text{ if }\ x=0,\\
\max\{\psi(x),\,-R\}&\text{ otherwise},
\end{cases}
\]
and observe that $\chi_r=\psi$ in  $\overbar\varOmega\setminus B_r$ and, for any  differentiability point of $\chi_r$ in
$\varOmega$, if $\chi_r(x)>-R$, then $\chi_r(x)=\psi(x)$ and $D\chi_r(x)=D\psi(x)$ and, if $\chi_r(x)=-R$, then $D\chi_r(x)=0$. 
Hence, 
\[
D\chi_r=\begin{cases}
D\psi&\text{ a.e. in }\ \varOmega\setminus B_r,\\
D\psi \ \ \text{ or } \ \ 0 &\text{ a.e. in }\ B_r.
\end{cases} 
\]
Let $\lambda>0$ be a constant to be 
fixed later, set $\psi_r:=\lambda\chi_r$ and note that, for a.e. $x\in\varOmega\setminus B_r$, 
\[
H(x,D\psi_r)\leq \theta^{-1}|D\psi_r|^2+b(x)\cdot D\psi_r 
\leq \lambda(\theta^{-1}C\lambda-1),
\] 
where  $C>0$ be a Lipschitz bound of \ $\chi_r$.
\smallskip

\noindent If $\lambda:=\theta/(2C)$, then  
\[
H(x,D\psi_r)\leq -\lambda/2 \ \ \ \text{ a.e. in } \ \varOmega\setminus B_r. 
\]
Similarly, it is easy to check that   
\[
H(x,D\psi_r)\leq 0  \ \ \ \text{ a.e. in } \ B_r\cap\varOmega.
\]
\end{proof}

\begin{cor}\label{cor:qp2} For all \ $y\in\overbar\varOmega\setminus\{0\}$,   \ \ $V_\varOmega(0,y)>0$. 
\end{cor}

\begin{proof} Fix any $y\in\overbar\varOmega\setminus\{0\}$, choose $r>0$ so that 
$y\not\in \overbar B_r$ and let $\psi_r\in\Lip(\overbar\varOmega)$ be as in the proof of Corollary 
\ref{cor:qp1}, so that 
\[ H(x,D\psi_r(x))\leq 0  \ \text{ a e. in } \ \varOmega \ \text{ and } \ 
\psi_r(0)<\psi_r(x). \]
Set $\phi:=-\psi_r$ and observe that 
\[
H(x,-D\phi)\leq 0 
\ \text{ a.e. in } \ \varOmega \ \text{ and }  \
\phi(0)-\phi(y)>0.
\]
Now, in view of  \eqref{eq:qp1}, 
\[
V_\varOmega(0,y)\geq \phi(0)-\phi(y)>0. \qedhere
\]
\end{proof} 

\noindent  The aim of the rest of this section is to construct a smooth approximation of $V(\cdot)=V_\varOmega(0,\cdot)$ which is a strict sub-solution
of the above Hamilton-Jacobi equation away from $0$ while it remains a sub-solution in the whole domain.

\begin{prop} \label{prop:qp2}For any $r>0$ there exist $V_r\in\Lip(\overbar\varOmega)$ and  
$\eta>0$ such that, 
\[
H(x, DV_r)\leq -\eta \ \text{ a.e. in } \ \varOmega\setminus B_r, \ 
H(x, DV_r)\leq 0 \ \text{ a.e. in } \ B_r \ \text{ and } \ 
\|V_r-V\|_{\infty,\varOmega}<r.
\]
\end{prop}

\begin{proof} Fix $r>0$, let $\psi_r\in\Lip(\overbar\varOmega)$ and $\eta>0$ be as in Corollary \ref{cor:qp1} and 
$\delta\in(0,\,1)$ a constant to be fixed later, define  $V_r\in\Lip(\overbar\varOmega)$ 
by $V_r:=(1-\delta)V+\delta\psi_r$ and observe that
\[
\begin{cases}
H(x,DV_r(x))\leq (1-\delta)H(x,DV(x))+\delta H(x,D\psi_r(x))\leq -\eta\delta  
\ \ \ \text{ a.e. in } \varOmega\setminus B_r, \\[3pt]
H(x,DV_r(x))\leq 0  \  \text{ a.e. in }  \  B_r,\\[3pt]
|V(x)-V_r(x)|\leq \delta|\psi_r(x)-V(x)| \ \ \ \text{ for all } \ x\in\overbar\varOmega.&
\end{cases}
\]
The claim follows if $\delta>0$ is so small that  
$\delta\|V-\psi_r\|_{\infty,\varOmega}<r$.  
\end{proof} 

\begin{thm}\label{thm:qp2} For any $r>0$ there exist $W_r\in C^2(\overbar\varOmega)$ and 
$\eta>0$ such that 
\[
H(x,DW_r)\leq -\eta \ \text{ in } \ \varOmega\setminus B_r, \ 
H(x,DW_r)\leq 1 \ \text{ in } \ B_r, \ \text{ and } \ 
\|W_r-V\|_{\infty,\varOmega}<r.
\]
\end{thm}  

\begin{proof} Fix $r>0$ and let $V_r\in\Lip(\overbar\varOmega)$ and $\eta>0$ be as in  Proposition 
\ref{prop:qp2} and  $\delta>0$. 
In view  of the $C^1$-regularity of $\partial\varOmega$, there exists 
a $C^1$-diffeomorphism $\Phi_\delta\mid \R^n\to\R^n$ such that 
$$ \Phi_\delta(\overbar\varOmega)\subset \varOmega, \  \|D\Phi_\delta-I\|_{\infty,\R^n}<\delta \  \text{and} \ \Phi(x)=x \  \text{ for all }\ x\in B_r.$$
Let $V_{r,\delta}:=V_r\circ\Phi_\delta$, observe that $V_{r,\delta}\in\Lip(U_\delta)$ where $U_\delta:=\Phi_\delta^{-1}(\varOmega)$,  
and fix $\delta>0$ sufficiently small so that 
\[
\begin{cases}
H(x,DV_{r,\delta})<-\eta/2&\text{ a.e. in } \ U_\delta\setminus B_r,\\[3pt]
H(x,DV_{r,\delta})\leq 0 &\text{ a.e. in } \ B_r\cap U_\delta, \\[3pt]
\|V_{r,\delta}-V\|_{\infty,\varOmega}<2r.
\end{cases}
\]
\smallskip

\noindent Next let $\rho$ be a standard mollifier 
in $\R^n$ with $\supp \rho \subset B_1$, 
and choose $\gamma\in(0,\,r/2)$ small enough so that $B_\gamma(x)\subset U_\delta$ 
for all $x\in\overbar\varOmega$. Hence, 
$W:=\rho_\gamma*V_{r,\delta}$ is well-defined in $\overbar\varOmega$, where 
$\rho_\gamma(x):=\gamma^{-n}\rho (\gamma^{-1} x)$. 
\smallskip

\noindent Let $L>0$  and  
$\omega_H$ be respectively a Lipschitz bound of $V_{r,\delta}$ and  
the modulus of continuity of $H$ on $\overbar\varOmega\times B_L$,
fix any $x\in\varOmega\setminus B_{2r}$, note that   
\[
-\frac\eta 2\geq H(x-y,DV_{r,\delta}(x-y))\geq H(x,DV_{r,\delta}(x-y))-\omega_H(\gamma) \ \ \
\text{ for a.e. }\ y\in B_{\gamma},
\] 
and observe that, by Jensen's inequality, 
\[
H(x,DW)
\leq \int_{B_\gamma}H(x,DV_{r,\delta}(x-y))\rho_\gamma(y)\di y\leq \omega_H(\gamma)-\frac{\eta}{2}.
\] 
Similarly, we find that, for any $x\in\varOmega$, 
\[
H(x,DW)\leq \omega_H(\gamma). 
\]
Thus, for  $\gamma>0$ small enough,  
\[
\begin{cases}
H(x,DW(x))<-\eta/3 &\text{ for all } \ x\in\varOmega\setminus B_{3r},\\[3pt]
H(x,DW(x))\leq 1 & \text{ for all } \ x\in B_{3r}\cap\varOmega,\\[3pt]
\|W-V\|_{\infty,\varOmega}<3r.&
\end{cases}
\]
The function $W$ has all the properties required for $W_{3r}$ and, since $r>0$ is arbitrary, this completes the proof. 
\end{proof}

 
\section{Asymptotics in a smaller time scale} \label{sec:sts}

\noindent Fix  $r>0$ and $\mu>0$ and let $W_r\in C^2(\overbar\varOmega)$ and $\eta>0$ be given by  Theorem \ref{thm:qp2}.  
For $\varepsilon>0$ and $x \in \overbar\varOmega$ set 
\begin{equation}\label{eq:sts1a}
v^\varepsilon(x):=\exp\left(\frac{W_r(x)-\mu}{\varepsilon}\right), 
\end{equation}
and note that
\[
L_\varepsilon v^\varepsilon=\frac{v^\varepsilon}{\varepsilon}\left(H(x, DW_r)+\varepsilon\tr[aD^2W_r]\right).
\]
\smallskip

\noindent Select  $C>0$ and $\varepsilon_0>0$ so that $\,\varepsilon_0C<1\wedge \eta$ and, for all  $\,x \in \overbar\varOmega\,$,
$\,|\tr[a(x)D^2W_r(x)]|\leq C$.
It follows that, for any $\varepsilon\in(0,\,\varepsilon_0)$,  
\begin{equation}\label{eq:sts2}
L_\varepsilon v^\varepsilon\leq \frac{v^\varepsilon}{\varepsilon}\left(H(x,DW_r)+\varepsilon C\right)<
\begin{cases}
0&\text{ in } \ \varOmega\setminus B_r,\\[3pt] \displaystyle
\frac{2v^\varepsilon}{\varepsilon}& \text{ in } \ B_r.
\end{cases}
\end{equation}
\smallskip

\noindent Set \ $R_\varepsilon:=(2/\varepsilon)\|v^\varepsilon\|_{\infty,B_r\cap\varOmega}$ \ and, for  $(x,t)\in\overbar Q$,
\begin{equation}\label{eq:sts3}
w^\varepsilon(x,t):=v^\varepsilon(x)+R_\varepsilon t. 
\end{equation}
It follows easily that, for any $\varepsilon\in(0,\,\varepsilon_0)$,
\begin{equation}\label{eq:sts4}
w_t^\varepsilon>L_\varepsilon w^\varepsilon\ \ \ \text{ in } \ Q. 
\end{equation}  
\smallskip

\noindent The main result of this section is about the behavior, as $\varepsilon\to 0$, of the solution 
$u^\varepsilon\in C(\overbar Q)\cap C^{2,1}(Q)$ of \eqref{eq:int1},
\eqref{eq:boundary} in $B_r \times [0,\exp \lambda/\varepsilon]$, for $\lambda$ as in the statement below.

\begin{thm}\label{thm:sts1} If  $\lambda>0$ is such that 
$\{V\leq \lambda\}\subset\{g\leq 0\}$, 
then, for each $\delta>0$, there exists $r>0$ such that
\[
\lim_{\varepsilon\to 0+} (u^\varepsilon-\delta)_+=0 \quad\text{ uniformly on } \ B_r\times [0,\e^{\lambda/\varepsilon}]. 
\]
\end{thm}

\begin{proof} Choose $r>0$ so small that $B_r\subset \varOmega$, let $W_r$, $\varepsilon_0$, $v^\varepsilon$ and $w^\varepsilon$ be as above,
fix $\delta>0$ and set $\,U^\varepsilon:=u^\varepsilon-\delta$ and $G:=g-\delta$. 
Since  $\{G\leq 0\}$ is a 
neighborhood of $\{g\leq 0\}$, we may choose $\gamma>\lambda$ such that 
\[
\{V\leq\gamma\}\subset\{G\leq 0\}.
\]  
It follows from the maximum principle that 
$\sup_{Q}U^\varepsilon\leq \sup_{\varOmega}G$ and, hence, 
\[
U^\varepsilon\leq M:=\|G\|_{\infty,\varOmega} \ \ \ \text{ on } \  \overbar Q.
\] 

\noindent Fix $\mu>0$ in \eqref{eq:sts1a} (the definition of $v^\varepsilon$) so that 
$\lambda<\mu<\gamma$, and, if needed, 
select $r>0$ even smaller so that $\gamma-r-\mu>0$, which ensures that 
\[
v^\varepsilon=\exp\left(\frac{W_r-\mu}{\varepsilon}\right)
>\exp\left(\frac{\gamma-r-\mu}{\varepsilon}\right) \ \text{ in } \ \{G>0\}.
\] 
Taking, if necessary, $\varepsilon_0>0$ even smaller, we may assume that, if $0<\varepsilon<\varepsilon_0$, then 
\[
\exp\left(\frac{\gamma-r-\mu}{\varepsilon}\right)>M. 
\] 
Hence, for $\varepsilon\in(0,\,\varepsilon_0),$
\[
v^\varepsilon>M\geq G \ \text{in} \ \{G> 0\} \  \text{ and } \ 
v^\varepsilon>0\geq G \  \text{ in } \  \{G\leq 0\},
\]
and, accordingly, 
$\,w^\varepsilon\geq G=U^\varepsilon\,$  on  $\,  \partial_{\mathrm{p}} Q$. 
Using the maximum principle we get, for all $(x,t)\in \overbar\varOmega\times[0,\,\infty)$ and $ \varepsilon\in(0,\,\varepsilon_0)$,
\begin{equation} \label{eq:sts5}
U^\varepsilon(x,t)\leq w^\varepsilon(x,t)=v^\varepsilon(x)+R_\varepsilon t. 
\end{equation}

\noindent Since $V\in\Lip(\overbar\varOmega)$ and $V(0)=0$, there exists $C_0>0$ such that 
$\,|W_r|\leq C_0r\,$ in $\, B_r$ 
and, therefore,
\[ v^\varepsilon\leq\exp\left(\frac{C_0r-\mu}{\varepsilon}\right) \ \ \ \text{ in  } \ B_r.\]

\noindent Next assume that  $r$ is even  smaller so that 
$\tilde\mu:=\mu-C_0r>\lambda$, which implies that  
\[
R_\varepsilon=\frac{2\|v^\varepsilon\|_{\infty,B_r}}{\varepsilon}< \frac{2\e^{-\tilde\mu/\varepsilon}}{\varepsilon},
\]
and, for all $x\in B_r$ and $0\leq t\leq \e^{\lambda/\varepsilon}$,
\[
w^\varepsilon(x,t)\leq \e^{-\tilde\mu/\varepsilon}+\frac{2\e^{-\tilde\mu/\varepsilon}}{\varepsilon}\,\e^{\lambda/\varepsilon}
=\e^{-\tilde\mu/\varepsilon}+\frac{2\e^{(\lambda-\tilde\mu)/\varepsilon}}{\varepsilon}. 
\] 
Hence,  
\[
\lim_{\varepsilon\to 0+} U^\varepsilon_+=0 \  \text{ uniformly on } \ B_r\times[0,\,\e^{\lambda/\varepsilon}]. \qedhere
\]
\end{proof}


\section{Asymptotics in a larger time scale}\label{sec:lts}

\noindent The main theorem concerns the behavior of the solutions $u^\varepsilon\in C(\overbar Q)\cap C^{2,1}(Q)$ of 
\begin{equation}\label{100}
\begin{cases}
u_t^\varepsilon=L_\varepsilon u^\varepsilon&\text{ in }\ Q,\\[3pt]
u^\varepsilon=0&\text{ on }\ \partial\varOmega\times(0,\infty),\\[3pt]\displaystyle 
\sup_{\varepsilon>0}\|u^\varepsilon\|_{\infty,Q}<\infty. &
\end{cases}
\end{equation}

\noindent We have:

\begin{thm}\label{thm:lts1} Fix $\lambda>m_0$ and, for $\varepsilon>0$, assume that  $u^\varepsilon\in C(\overbar Q)\cap C^{2,1}(Q)$
solves \eqref{100}.
Then 
\[
\lim_{\varepsilon\to 0+}u^\varepsilon=0 \quad\text{ uniformly on }\ \overbar\varOmega\times[\e^{\lambda/\varepsilon},\,\infty).
\]
\end{thm}

\noindent The following proposition is a key observation needed to prove Theorem \ref{thm:lts1}.  Its proof is presented later in the section.

\begin{prop}\label{prop:lts1} Let $\lambda>m_0$. There exists $W\in \Lip(\overbar\varOmega)$ 
and  $\eta>0$ such that
\begin{equation}
0<\min_{\overbar\varOmega}W\leq \max_{\overbar\varOmega}W<\lambda,
\end{equation}
and, in the viscosity super-solution sense,  
\begin{equation}\label{eq:lts0}
H(x,-DW)\geq \eta \ \text{ in } \ \varOmega \ \text{and} \  \
\eta\tr[a(x)D^2W(x)]\leq 1 \text{ in } \ \varOmega.
\end{equation}
\end{prop}


\begin{proof}[Proof of Theorem \ref{thm:lts1}] 
Since $u^\varepsilon$ and $-u^\varepsilon$ both solve \eqref{100}, 
it is enough to show that, for any $\lambda>m_0$, 
\begin{equation}\label{eq:lts1}
\lim_{\varepsilon\to 0+}u^\varepsilon_+=0 \ \text{ uniformly on }\ \overbar\varOmega\times[\e^{\lambda/\varepsilon},\,\infty).
\end{equation}
Moreover, multiplying the $u^\varepsilon$'s  by a positive constant if necessary, we may assume that 
\[
\sup_{\varepsilon>0}\|u^\varepsilon\|_{\infty,Q}\leq 1. 
\]
\smallskip

\noindent Fix any $\lambda>m_0$. 
Let $W\in \Lip(\overbar\varOmega)$ and $\eta>0$ be as in Proposition \ref{thm:lts1} and set 
\[
\delta:=\min_{\overbar\varOmega}W\quad\text{ and }\quad \mu:=\max_{\partial\varOmega}W. 
\]
For $\varepsilon\in(0,\,\eta^2/2)$, set 
\[
v^\varepsilon(x):=\exp\left(-\frac{W(x)}{\varepsilon}\right) \quad\text{ for }\ x\in\overbar\varOmega,
\]
observe, using  \eqref{eq:lts0}, that, in the sub-solution sense, 
\[
\begin{cases}\displaystyle
\frac{\varepsilon^2}{v^{\varepsilon,2}}\,aDv^\varepsilon\cdot Dv^\varepsilon
+\frac{\varepsilon}{v^\varepsilon}\,b\cdot Dv^\varepsilon
\geq\eta &\text{ in }\ \varOmega,\\[3pt]\displaystyle\\[.25mm]
-\frac{\varepsilon}{v^\varepsilon}\tr(aD^2v^\varepsilon)+\frac{\varepsilon}{v^{\varepsilon,2}}\,aDv^\varepsilon\cdot Dv^\varepsilon\leq \frac 1\eta
&\text{ in }\ \varOmega,
\end{cases}
\]
and, consequently,
\[
L_\varepsilon v^\varepsilon\geq \frac{v^\varepsilon}{\varepsilon}\left(-\frac{\varepsilon}{\eta}+\eta\right)
\geq \frac{\eta v^\varepsilon}{2\varepsilon}
 \ \ \ \text{ in }\ \varOmega.
\]
Note also that 
\[
\e^{-\mu/\varepsilon}\leq v^\varepsilon\leq \e^{-\delta/\varepsilon} \  \text{in  } \ \overbar\varOmega. 
\]

\noindent Next we fix some $\gamma\in(0,\,\eta]$, set,  for $ (x,t)\in\overbar Q$, 
\[
w^\varepsilon(x,t):=1+\e^{-\delta/\varepsilon}-v^\varepsilon(x)-\frac{\gamma}{2\varepsilon}\e^{-\mu/\varepsilon}\,t
\]
and observe that 
\[
w_t^\varepsilon-L_\varepsilon w^\varepsilon
=-\frac{\gamma}{2\varepsilon}\e^{-\mu/\varepsilon}+L_\varepsilon v^\varepsilon
\geq -\frac{\gamma}{2\varepsilon}\e^{-\mu/\varepsilon}+\frac{\eta v^\varepsilon}{2\varepsilon}\geq 0 \ \ \ \text{ in }\ Q,
\]
and  
\[\begin{cases}
w^\varepsilon(x,0)\geq 1 &\text{ for all } \ x\in\overbar\varOmega,\\[3pt]\displaystyle
w^\varepsilon(x,t)\geq 1-\frac{\gamma}{2\varepsilon}\e^{-\mu/\varepsilon}\,t&\text{ for all }\ (x,t) \in\partial\varOmega\times [0,\,\infty), \\[3pt]\displaystyle
w^\varepsilon(x,t)\leq 1+\e^{-\delta/\varepsilon}-\frac{\gamma}{2\varepsilon}\e^{-\mu/\varepsilon}\,t&\text{ for all }\ (x,t) \in\overbar\varOmega\times [0,\,\infty).  
\end{cases}
\]
Then for 
$T:=\frac{2\varepsilon}{\gamma}\,\e^{\mu/\varepsilon},$
we have 
\[
u^\varepsilon\leq w^\varepsilon\ \ \ \text{ on  } \  (\overbar \varOmega\times\{0\})\ \cup
\ (\partial\varOmega\times(0,\,T)),
\]
and, by the comparison principle, 
\[
u^\varepsilon\leq w^\varepsilon\ \ \ \text{ on } \ \overbar \varOmega\times[0,\,T],
\]
and, in particular,  
\[
u^\varepsilon(x,T)\leq \e^{-\delta/\varepsilon} \ \ \ \text{ for all }\ x\in\varOmega. 
\]
Since $\gamma\in(0,\,\eta]$ is arbitrary, it follows that 
\[
u^\varepsilon\leq \e^{-\delta/\varepsilon} \  \text{ on  } \  \overbar \varOmega\times[(2\varepsilon/\eta)\e^{\mu/\varepsilon},\,\infty),
\]
from which we conclude that \eqref{eq:lts1} holds.  
\end{proof}

\noindent The proof of Proposition \ref{prop:lts1} requires a number of technical facts which we state and prove first. To this end, 
we introduce the function $U\in\Lip(\overbar\varOmega)$ given, for 
$x\in\overbar\varOmega$,  by 
\[ 
U(x):=\inf\{V_\varOmega(x,y)\mid y\in\partial\varOmega\}
=\min\{V_\varOmega(x,y)\mid y\in\partial\varOmega\}. 
\] 
Indeed, since $\,H(x,-D_xV_\varOmega(x,y))\leq 0$ for a.e. $x\in\varOmega$,  
the collection 
$\{V_\varOmega(\cdot \,,y): y\in\partial\varOmega\}$ is equi-Lipschitz continuous on $\overbar\varOmega$. Hence, $U$ 
is Lipschitz continuous on $\overbar\varOmega$ and, for each $x\in\overbar\varOmega$, 
the minimum in the above formula is achieved at a point $y\in\partial\varOmega$. 
It follows 
from the theory of viscosity solutions 
that $U$ is a super-solution of 
\begin{equation}\label{eq:lts2}
H(x,-DU)=0 \ \ \ \text{ in } \ \varOmega, 
\end{equation}
and, moreover, satisfies $$H(x,-DU(x))\leq 0 \ \text{a.e. in
 $\varOmega$,}$$ 
which ensures that $U$ is a solution of \eqref{eq:lts2}.

\begin{lem}\label{lem:lts1}The function $U$ is the maximal sub-solution of 
\begin{equation}\label{eq:lts3}
\begin{cases}
H(x,-Du)=0 &\text{ in } \ \varOmega,\\[3pt]
u=0 &\text{ on } \ \partial\varOmega,
\end{cases}\end{equation}
and 
\[
0\leq U\leq U(0)=m_0 \ \text{ on } \ \overbar \varOmega.
\]
\end{lem}

\begin{proof} 
Let $u \in C(\overbar\varOmega)$ be a sub-solution of \eqref{eq:lts3}. By \eqref{eq:qp1},
we have
\[
u(x)\leq V_\varOmega(x,y) \ \ \ \text{ for all } \ x\in\overbar\varOmega,\,y\in\partial\varOmega, 
\] 
and, hence, 
\[
u\leq U \ \text{ on } \ \overbar\varOmega. 
\]
Since $U$ is a solution of \eqref{eq:lts3}, this last observation yields the first part of the claim.
\smallskip

\noindent  Moreover,  
\[
U(0)=\inf\{V_\varOmega(0,y)\mid y\in\partial\varOmega\}=\min_{\partial\varOmega}V=m_0. 
\]
Next we recall that $V_\varOmega(x,0)=0$ for all $x\in\overbar\varOmega$. Hence,
\[
V_\varOmega(x,y)\leq V_\varOmega(x,0)+V_\varOmega(0,y)=V_\varOmega(0,y) \ \ \ \text{ for all } \ y\in\overbar\varOmega.
\]
Taking infimum over all $y\in\partial\varOmega$, we find
$U\leq U(0) \  \text{ on } \  \overbar \varOmega$, and the proof is now complete. 
\end{proof}

\begin{lem}\label{lem:lts2} For each $\gamma>0$, there exists a unique solution $u\in\Lip(\overbar\varOmega)$ 
of 
\begin{equation}\label{eq:lts4}
\begin{cases} 
H(x,-Du)=\gamma &\text{ in }\ \varOmega,\\[3pt]
u=0 &\text{ on } \ \partial\varOmega. 
\end{cases}
\end{equation}
\end{lem}

\begin{proof}  Choose $M>0$ such that 
\[
H(x,p)\geq \gamma \  \text{ for all } \ (x,p)\in\overbar\varOmega\times(\R^n\setminus B_M).  
\] 
It is easy to check that 
$f(x):=M\dist(x,\partial\varOmega) $
is a super-solution of \eqref{eq:lts4}. 
It is also obvious that 
$0$ is a sub-solution of \eqref{eq:lts4}. Perron's method now implies  that 
there exists a solution $u\in\Lip(\overbar\varOmega)$ of \eqref{eq:lts4}. 
\smallskip

\noindent  Note that $H(x,0)=0<\gamma$ for all $x\in\varOmega$ and recall that $p\mapsto H(x,p)$ is 
convex for any $x\in\varOmega$. Under these conditions, the uniqueness follows from a well known comparison (see e.g. 
\cite{BC1997,Ba1994, Is1987}) 
which we state below as a separate lemma without proof. 

\end{proof} 

\begin{lem}\label{comparison} Let $\gamma>0$. 
If $u\in C(\overbar\varOmega)$ (resp. $v\in C(\overbar\varOmega)$) is a sub-solution (resp. 
super-solution) of $H(x,-Dw(x))=\gamma$ in $\varOmega$ and $u\leq v$ on $\partial\varOmega$, then 
$u\leq v$ in $\varOmega$. 
\end{lem}

\noindent We continue with 
\begin{lem}\label{lem:lts3} For each $\gamma>0$ let $u_\gamma\in\Lip(\overbar\varOmega)$ be the solution 
of \eqref{eq:lts4}. Then 
\begin{equation}\label{eq:lts5}
\lim_{\gamma\to 0}u_\gamma=U  \ \text{ uniformly on } \ \overbar\varOmega. 
\end{equation}
\end{lem} 

\begin{proof} Note that if $0<\gamma_1<\gamma_2$, then $u_{\gamma_1}$ is a subsolution of 
\eqref{eq:lts4} with $\gamma=\gamma_2$. Therefore the comparison yields 
\[
u_{\gamma_1}\leq u_{\gamma_2} \  \text{ on   } \  \overbar\varOmega. 
\]
Observe also that the $u_\gamma$'s , with $\gamma\in(0,\,1)$, are sub-solutions to  \eqref{eq:lts4} 
with $\gamma=1$ and, therefore,  the collection $\{u_\gamma\}_{\gamma\in(0,\,1)}$ is equi-Lipschitz 
on $\overbar\varOmega$.  It follows that there exists some $u\in\Lip(\overbar\varOmega)$ such that  
\begin{equation}\label{eq:lts6}
\lim_{\gamma\to 0}u_\gamma=u \  \text{ uniformly on  } \ \overbar\varOmega \  \text{ and } \  u=0  \ \text{ on } \  \partial\varOmega.
\end{equation}  

\noindent The stability  of viscosity solutions yields  that $u$ is a solution of \eqref{eq:lts3},
and, moreover, by the maximality of $U$, that $U\geq u$ on $\overbar\varOmega$. 
\smallskip

\noindent Note also that $U$ is a sub-solution of \eqref{eq:lts4} with $\gamma>0$. Hence,
$U\leq u_\gamma \  \text{ on } \ \overbar\varOmega$ and, therefore, 
and $U\leq u$ on  $\overbar\varOmega$.
\smallskip

\noindent Thus we conclude that 
$u=U$ on $\overbar\varOmega$ and \eqref{eq:lts5} holds.     
\end{proof} 

\noindent We are now in a position to present the 

\begin{proof}[Proof of Proposition \ref{prop:lts1}] Fix $\gamma>0$. It follows from Lemma \ref{lem:lts3} that, 
if $\mu \in (0,\gamma)$ is sufficiently small, then the solution $u_\mu\in\Lip(\overbar\varOmega)$ of \eqref{eq:lts4},  
with $\gamma$ replaced by $\mu$, satisfies 
\[
\|U-u_\mu\|_{\infty,\varOmega}<\gamma,
\]
and, moreover, $0\leq u_\mu(x)<m_0+\gamma$ for all $x\in\overbar\varOmega$. 
\smallskip

\noindent  For $x\in\overbar\varOmega$ set 
\[
W(x):=u_\mu(x)+\gamma, 
\]
fix any $\delta>0$ and, as in the proof of Theorem \ref{thm:qp2}, choose a $C^1$-diffeomorphism\\
$\Phi_\delta:\R^n\to\R^n$ so that 
\[
\Phi_\delta(\overbar\varOmega)\subset\varOmega, \  
\|D\Phi_\delta(x)-I\|_{\infty,\R^n}<\delta \ \text{and}  \ \Phi(0)=0,
\]
Let
\[
W_\delta:=W\circ\Phi_\delta
\] 
and note that $\Phi_\delta^{-1}(\varOmega)$ is an open neighborhood of $\overbar\varOmega$. 
We deduce that, if $\delta>0$ is sufficiently small, then 
\[
H(x,-DW_\delta)\geq \mu/2  \ \text{ in } \ \Phi_\delta^{-1}(\varOmega) \ \text{ and} \ \gamma\leq W_\delta\leq m_0+2\gamma \ \text{ on } \ \overbar\varOmega.
\]

\noindent For $\alpha>0$ small, we introduce the inf-convolution $W_{\delta,\alpha}$ of  $W_\delta$, given, for $x\in\R^n$, by
\[
W_{\delta,\alpha}(x):=\inf\{W_\delta(y)+\frac{1}{\alpha}|x-y|^2\mid y\in\Phi_\delta^{-1}(\varOmega)\}. 
\]
As is well-known (see, for example, \cite{Ba1994}), $W_{\delta,\alpha}$ is semi-concave in $\Phi_\delta^{-1}(\varOmega)$, that is
\[
\max\{D^2W_{\delta,\alpha}(x)\xi\cdot\xi\mid \xi\in B_1\}\leq C_{\delta,\alpha} \ \ \ \text{ in } \Phi_\delta^{-1}(\varOmega)
\]
holds in the super-solution sense for some constant $C_{\delta,\alpha}$, depending on $\delta,\,\alpha$,
and, if $\alpha>0$ is sufficiently small, then 
\[
\|W_\delta-W_{\delta,\alpha}\|_{\infty,\varOmega}<\gamma/2 \ \ \  \text{and} \ \ \  H(x,-DW_{\delta,\alpha})\geq \mu/4   \ \text{ in} \ \varOmega,\]
where the latter inequality holds in the super-solution sense. 
It is then easily checked that $W_{\delta,\alpha}$ satisfies, in the super-solution sense, 
\[
\tr \left[aD^2W_{\delta,\alpha}\right]\leq C_{\delta,\alpha}\tr a \ \ \ \text{ in }\ \varOmega.
\]
Thus, noting that 
\[
\gamma/2\leq W_{\delta,\alpha}\leq m_0+3\gamma \ \ \ \text{ in } \  \varOmega,
\]
and choosing $\,\gamma>0\,$ and $\,\eta>0\,$ so small that 
$\,m_0+3\gamma<\lambda$, 
$\,\eta C_{\delta,\alpha}\|\tr a\|_{\infty,\varOmega}\leq 1\,$ 
and $\,\eta \leq \mu/4$,
we conclude that $W:=W_{\delta,\alpha}$ and $\eta$ have the required properties. 
\end{proof}


\section{The Stationary problem}  \label{sec:sp}

\noindent We consider the Dirichlet problem
\begin{equation}\label{eq:sp1}
\begin{cases}
L_\varepsilon v^\varepsilon=0 &\text{ in } \ \varOmega,\\[3pt]
v^\varepsilon=g &\text{ on }\ \partial\varOmega,
\end{cases}
\end{equation}
for
\begin{equation}\label{g} 
g\in C(\overbar\varOmega) \ \text{ such that} \ g=0  \ \text{ on } \ \argmin(V|\partial\varOmega).
\end{equation}

\noindent The next result is an essential part of a classical observation obtained by  
Freidlin-Wentzell \cite{FW2012}, Devinatz-Friedman \cite{DF1978}, Kamin \cite{Ka1978,Ka1979}, Perthame \cite{Pe1990}, etc..  


\begin{thm}\label{thm:sp1} Assume (A1)--(A5) and \eqref{g}.  Then 
$\lim_{\varepsilon\to 0+}v^\varepsilon(0)=0.$
\end{thm}


\begin{proof} We show that, for any $\delta>0$, there exits $r>0$ such that 
\[
\lim_{\varepsilon\to 0+}(v^\varepsilon-\delta)_+=0\quad\text{ uniformly on }\ B_r.
\]
Applying the above claim to the pair $(-v^\varepsilon,-g)$ in place of $(v^\varepsilon,g)$ yields  that, for any $\delta>0$, 
there exist $r>0$ such that 
\[
\lim_{\varepsilon\to 0+}(-v^\varepsilon-\delta)_+=0 \  \text{ uniformly on } \ B_r,
\]
and, hence, the desired conclusion. 
\smallskip
 
\noindent Now,  fix  $\delta>0$, observe that  $\{x\in\overbar\varOmega\mid g(x)< \delta\}$ is 
a neighborhood, relative to $\overbar\varOmega$, of \ $\argmin(V|\partial\varOmega)$,  choose $\lambda>\mu>m_0$ and $h\in C(\overbar\varOmega)$ so that
\[
G:=g-\delta< 0 \  \text{ on } \ \{V\leq \lambda\}\cap\partial\varOmega,
\] 
and 
\[ 
h=0  \ \text{ on } \  \partial\varOmega, \ 
0\leq h\leq G_+  \ \text{ on } \ \overbar\varOmega, \ \text{ and } \ 
h=G_+ \ \text{ in } \ \{V\leq \mu\}. 
\]
\begin{center}
\includegraphics[height=4.6cm]{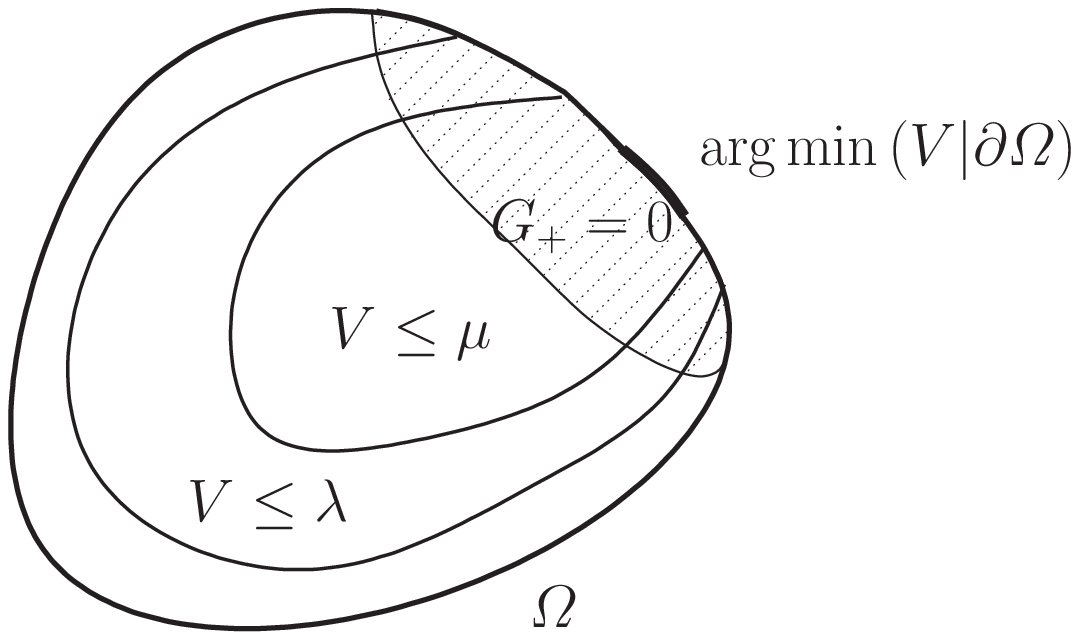}
\end{center}
Let $u^\varepsilon\in C(\overbar Q)$ be the solution of
\[
\begin{cases}
u_t^\varepsilon=L_\varepsilon u^\varepsilon&\text{ in }\ Q,\\[3pt]
u^\varepsilon=h &\text{ on  } \ \partial_{\mathrm{p}}Q. 
\end{cases}
\]
Theorem \ref{thm:lts1} gives  
\[
\lim_{\varepsilon\to 0+}u^\varepsilon=0\quad\text{ uniformly on } \ 
\overbar\varOmega\times[\e^{\mu/\varepsilon},\,\infty).
\]
For  $(x,t)\in\overbar Q$ set
\[
w^\varepsilon(x,t):=v^\varepsilon(x)-\delta-u^\varepsilon(x,t)
\]
and note that 
 \[\{V\leq \mu\}\ \subset \ \{ G-h\leq 0\} \ \ \ \text{ and } \ \  \ w^\varepsilon=G-h \ \ \text{ on }\ 
 \partial_{\mathrm{p}} Q.
 \]
It follows from  Theorem \ref{thm:sts1} that there exists  $r>0$ such that
 \[
 \lim_{\varepsilon\to 0+}(w^\varepsilon-\delta)_+=0\quad\text{ uniformly on } \ B_r\times[0,\,\e^{\mu/\varepsilon}]. 
 \]
 Since, for all $x\in\varOmega$, 
 \[
 (v^\varepsilon(x)-2\delta)_+\leq u^\varepsilon(x,\e^{\mu/\varepsilon})_++(w^\varepsilon(x,\e^{\mu/\varepsilon})-\delta)_+,  
 \]
 we find  that  
 \[
 \lim_{\varepsilon\to 0+}(v^\varepsilon-2\delta)_+=0 \ \text{ uniformly on } \ B_r. \qedhere
 \]
\end{proof}

\noindent We have indeed shown the following: 

\begin{thm}\label{thm:sp2} Assume (A1)--(A5) and \eqref{g}.  
For any $\delta>0$ there exists $r>0$
such that 
\[\lim_{\varepsilon\to 0}(v^\varepsilon-\delta)_+=0 
\ \ \ \text{ uniformly on }\ B_r. \]
\end{thm}


\section{Asymptotic constancy}\label{sec:ac}
\noindent In this section we state precisely the claim that the limit, as $\varepsilon\to 0$, of \eqref{eq:int1} 
is the transport equation \ $u_t=b\cdot Du$ and provide its proof where 
(A5) plays a critical  role.  
 
\begin{thm}\label{thm:ac1} Let $\tau(\varepsilon) >0$ be  
such that $\,\lim_{\varepsilon\to 0+}\tau(\varepsilon)=\infty\,$ 
and for each $\varepsilon>0$, let $u^\varepsilon\in C(\overbar Q)\cap C^{2,1}(Q)$ 
be a solution of \eqref{eq:int1}. 
Assume that, for some constant $r>0$, 
\begin{equation}\label{eq:ac1}
\lim_{\varepsilon\to 0+} u^\varepsilon=0 \quad\text{ uniformly on }\  B_r\times[0,\,\tau(\varepsilon)),
\end{equation}
and 
\[
\sup_{\varepsilon>0}\|u^\varepsilon\|_{\infty,\, \varOmega\times[0,\,\tau(\varepsilon))}<\infty. 
\]
There exists $T=T(r)>0$ such that, for any compact subset $K$ of  $\varOmega$ and any $\tau_0>0$, 
\[
\lim_{\varepsilon\to 0+}u^\varepsilon=0\quad\text{ uniformly on }\ K\times[T,\,\tau(\varepsilon)-\tau_0). 
\]
\end{thm}

\noindent  As before  we prove a slightly generalized,  
one-sided version of the above theorem, which readily yields the claim.

\begin{thm}\label{thm:ac2} Let $\tau(\varepsilon)>0$ 
be  such that \ $\lim_{\varepsilon\to 0+}\tau(\varepsilon)=\infty$, and for each $\varepsilon>0$, 
consider a solution  $u^\varepsilon\in C(\overbar Q)\cap C^{2,1}(Q)$ to \eqref{eq:int1}. Fix $r>0$ so that $B_r\subset \varOmega$ and let $N$ be a (possibly empty) 
open subset of $\partial\varOmega$. 
Assume that 
\begin{equation}\label{eq:ac2}
\lim_{\varepsilon\to 0+} u^\varepsilon_+ =0 \quad\text{ uniformly on } \  \left(B_r \ \cup \ N \right)\times[0,\,\tau(\varepsilon)),
\end{equation}
and  
\begin{equation}\label{eq:ac3}
\limsup_{\varepsilon\to 0+} \|u^\varepsilon\|_{\infty,\, \varOmega\times[0,\,\tau(\varepsilon))}<\infty
\end{equation}
There exists $T=T(r)>0$ such that, for any compact subset $K$  of  $\varOmega\,\cup\, N$ 
and any $\tau_0>0$, 
\[
\lim_{\varepsilon\to 0+}u^\varepsilon_+ =0 \quad\text{ uniformly on } \ K\times[T,\,\tau(\varepsilon)-\tau_0). 
\]
\end{thm}

\noindent The following lemma plays an important role in the proof of  Theorem 13. 

\begin{lem}\label{lem:ac1} Let $u\in\USC(\overbar Q)$ be a sub-solution of
$\,u_t =b\cdot Du\,$  in  $\,Q\,$ and, for 
$(x,t)\in Q$, set $X(s):=X(s;x)$.  
The function $s\mapsto u(X(s),t-s)$ is nondecreasing on $[0,\,t]$.  
\end{lem} 

\begin{proof} Note that  (A5) yields that,  for all $x\in\varOmega,\, s\geq 0$,  $X(s;x)\in\varOmega$. For $s\in[0,\,t]$,  
set $v(s):=u(X(s),t-s)$.  We show that,  in the sub-solution sense, 
$v'\geq 0$ in $(0,\,t)$
which implies that $v$ is nondecreasing on 
$[0,\,t]$.
\smallskip

\noindent Let $\phi\in C^1([0,\,t])$ and assume that $v-\phi$ attains a strict maximum at a point 
$\hat s\in(0,\,t)$. For $\alpha>0$ consider the map
\[
(y,s)\mapsto u(y,t-s)-\phi(s)-\alpha|y-X(s)|^2 
\]
on $\overbar\varOmega\times[0,\,t]$ and let $(y_\alpha,s_\alpha)$ be a maximum point. 
It is easy to see that, as $\alpha \to \infty$, 
$(y_\alpha,s_\alpha) \to (\hat s, X(\hat s))$ and  
$\alpha|y_\alpha-X(s_\alpha))|^2 \to 0$. 
Fix a sufficiently large $\alpha$ so that $(y_\alpha,s_\alpha)\in Q$. 
Noting that $\psi(y,s):=u(y,t-s)$ is a sub-solution of 
$-\psi_s=b\cdot D\psi$ in $\varOmega\times(0,\,t)$, we find that
\[
-\phi'(s_\alpha)+2\alpha(y_\alpha-X(s_\alpha))\cdot \dot X(s_\alpha)\leq 2\alpha b(y_\alpha)\cdot (y_\alpha-X(s_\alpha)),
\]
from which we get 
\[
\phi'(s_\alpha)\geq 2\alpha(y_\alpha-X(s_\alpha))\cdot(b(X(s_\alpha))-b(y_\alpha))
\geq -2\alpha L|y_\alpha-X(s_\alpha)|^2,
\]
where $L>0$ is a Lipschitz bound of $b$. Sending $\alpha\to \infty$ yields 
\ $\phi'(\hat s)\geq 0$ and the proof is complete.
\end{proof} 

\noindent We continue with the 
\begin{proof}[Proof of Theorem \ref{thm:ac2}] 
We introduce the upper relaxed limit $U\in\USC(\overbar Q)$ given by
\[
U(x,t):=\lim_{\lambda\to 0+}\sup\{u^\varepsilon(y,s)_+:(y,s) \in \overbar Q,\ |y-x|+|s-t|\leq \lambda,\ 
0<\varepsilon<\lambda\},
\]
and recall the standard observation that  $U$ is a sub-solution of 
$\,U_t=b\cdot DU\,$  in $\, Q.$
\smallskip

\noindent According to Lemma \ref{lem:qp1}, we may choose $T=T(r)>0$ such that, for all $(x,t)\in\overbar\varOmega \times [T,\infty)$, 
$X(s;x)\in B_r\,$. 
From Lemma \ref{lem:ac1} it follows that, for any $(x,t)\in Q$ and $s\in[0,\,t]$,  
$\,U(X(s;x),t-s)\geq U(x,t)$. Hence, 
for any $(x,t)\in Q$ with  $t\geq T$, we have $\,X(T;x)\in B_r$ and  
\begin{equation}\label{102}
U(x,t)\leq U(X(T;x),t-T)\leq 0. 
\end{equation}

\noindent  Next we show that 
\begin{equation}\label{eq:ac4}
U=0 \ \text{ on } \ N\times (T, \infty).
\end{equation} 
Fix $(y,s)\in N\times(T,\,\infty)$ and, in view of (A5), choose $R>0$ so small that 
\[
y+\lambda b(y)\in\varOmega \ \text{ for all } \ \lambda\in(0,\,R), \ \ 
s>R+T, \ \ \text{ and } \ \ 
\overbar B_R(y)\cap \partial\varOmega\subset N.
\]
Reformulating the last observation in terms of $\,l(y,s):= 
\{(y,s)+\lambda(b(y),-1)\mid \lambda>0\}$,
a half-line in $\R^{n+1}$ with vertex at $(y,s)$,
we have 
\begin{equation}\label{eq:ac4-5}
\begin{cases}
\overbar B_R(y,s)\cap l(y,s)\subset \varOmega\times (T,\,\infty),\\[3pt]
\overbar B_R(y,s)\subset \R^n\times(T,\,\infty), \\[3pt]
\overbar B_R(y,s)\cap \left(\partial\varOmega\times\R \right)\subset N\times(T,\,\infty).
\end{cases}
\end{equation}
 For any $\gamma\in(0,1)$, we consider the open convex cone in $\R^{n+1}$ with vertex at the origin given by 
\[
C_\gamma:=\bigcup_{\lambda>0}\lambda\left((b(y),-1))+B_\gamma\right),
\] 
and we set $$C_\gamma(y,s):=(y,s)+C_\gamma.$$
From (A5) again, 
we may choose $\gamma\in (0,\,1/2)$ small enough so that
\begin{equation}\label{eq:ac5}
\overbar B_R(y,s)\cap C_{2\gamma}(y,s)\subset \varOmega\times (T,\,\infty), 
\end{equation}
which strengthens the first inclusion of \eqref{eq:ac4-5}.  
Noting that $C_\gamma$ is an open neighborhood of $(b(y),-1)$, we may also choose 
$\rho\in(0,\,R)$ so that 
\[
(b(x),-1)\subset C_\gamma\quad\text{ for all }\ x\in B_\rho(y), 
\]
which ensures that 
\begin{equation}\label{eq:ac6}
(b(x),-1)\subset C_\gamma\quad\text{ for all }\ (x,t)\in B_\rho(y,s). 
\end{equation}
\smallskip

\noindent Define next $d,\,\phi\mid\R^{n+1}\to \R$ by 
\[
d(x,t):=\dist((x,t),\,C_\gamma(y,s)) \ \  \ \text{ and } \ \ \ \phi:=d^2. 
\]
It is well-known that $\phi \in C^1(\R^{n+1})$, 
$D\phi \in \Lip(\R^{n+1})$ and 
$D\phi(x,t)$ is in the (negative) 
dual cone of $C_\gamma$, i.e.,
\[
D\phi(x,t)\cdot (\xi,\tau)\leq 0 \ \ \ \text{ for all }\ 
(\xi,\tau)\in C_\gamma,\ (x,t)\in \R^{n+1}.
\] 
Combining the above remark  with \eqref{eq:ac6} yields
\begin{equation}\label{eq:ac7}
b\cdot D\phi\leq \phi_t  \ \text{ in } \ B_\rho(y,s). 
\end{equation}

\noindent Next we compare $u^\varepsilon$ and $\phi$ on the set
\[
Q(y,s):=B_\rho(y,s)\cap Q,  
\]
and note that 
\[
\partial Q(y,s)\subset (\partial B_\rho(y,s)\cap Q)\cup (\overbar B_\rho(y,s)\cap\partial Q).
\]


\noindent In view of \eqref{eq:ac5}, we may choose $\lambda>0$ so that 
\[
(\overbar C_\gamma(y,s)+\overbar B_\lambda)\cap \partial B_\rho(y,s)\subset \varOmega\times (T,\,\infty).
\]

\noindent Set 
\[
K:=(\overbar C_\gamma(y,s)+\overbar B_\lambda)\cap \partial B_\rho(y,s),
\]
which is clearly a compact subset of $\varOmega\times(T,\,\infty)$, and 
fix 
any $\delta>0$. 
\smallskip

\noindent  Note that   \eqref{102}  and \eqref{eq:ac3} imply that 
there exist $\varepsilon_0>0$  and $M>0$ such that, for all $\varepsilon\in(0,\,\varepsilon_0),$ 
\ $\overbar Q(y,s)\subset \overbar\varOmega\times (T,\,\tau(\varepsilon))$  and  
\begin{equation}\label{eq:ac8}
u^\varepsilon\leq \delta  \ \text{ in  } \ K  \ \text{ and } \ u^\varepsilon\leq M \ \ \text{ in } \ \overbar Q(y,s). 
\end{equation}
\noindent Set $A:=M/\lambda^2$. Then, for  $\varepsilon\in(0,\,\varepsilon_0)$, 
\begin{equation}\label{eq:ac9}
u^\varepsilon\leq M=A\lambda^2\leq A\phi  \text{ in } \  \overbar Q(y,s) \cap \{d\geq\lambda\}. 
\end{equation}

\noindent Since $\overbar B_\rho(t,s)\cap \partial Q$ is a compact subset of $N\times (T,\,\infty)$, in view of 
 \eqref{eq:ac2}, we may assume, replacing, if needed, $\varepsilon_0$ by a smaller positive number, 
that, for all $\varepsilon\in (0,\,\varepsilon_0)$, 
\begin{equation}\label{eq:ac10}
u^\varepsilon\leq \delta  \  \text{ in } \  \overbar B_\rho(t,s)\cap \partial Q. 
\end{equation}

\noindent Fix $(x,t)\in \partial Q(y,s)$ and $\varepsilon\in(0,\,\varepsilon_0)$. 
If $(x,t)\in \partial B_\rho(y,s)\cap Q=\partial B_\rho(y,s)\cap (\varOmega\times(T,\,\infty))$ and  
$(x,t)\not\in K$, then $d(x,t)\geq \lambda$ and, by \eqref{eq:ac9}, $u^\varepsilon(x,t)\leq A\phi(x,t)$. 
Otherwise, that is, if $(x,t)\in \partial B_\rho(y,s)\cap Q \cap K$,   \eqref{eq:ac8} gives \ 
$u^\varepsilon(x,t)\leq \delta$. 
\vskip.07in
\noindent Moreover, if $(x,t)\in \overbar B_\rho(y,s)\cap \partial Q$, then, by \eqref{eq:ac10}, we have   
\ $u^\varepsilon(x,t)\leq \delta$, and, therefore, for all  $\varepsilon\in(0,\,\varepsilon_0)$, 
\[
u^\varepsilon\leq \delta+A\phi  \ \text{ on  } \ \partial Q(y,s).
\]
   
\noindent Since, for each $t$,  $D\phi(\cdot,t)\in\Lip(\R^{n})$, 
there exists some  $C>0$ so that,  in the super-solution sense,
$\tr[aD_x^2\phi]\leq C \ \text{ in } \ Q(y,s),$
Hence, using \eqref{eq:ac7}, we see that 
$\psi(x,t):=\delta+A\phi(x,t)+\varepsilon AC t\,$ is a 
super-solution to
\[
\psi_t\geq L_\varepsilon\psi   \ \text{ in }\ Q(y,s). 
\] 
Thus, by comparison, we get 
\[
u^\varepsilon\leq \psi  \  \text{ in } \ Q(y,s), 
\]
which yields  
\[
U(y,s)\leq \delta+A\phi(y,s)=\delta,
\]
and, after letting $\delta \to 0$,  $U(y,s)=0$. 
This proves \eqref{eq:ac4}.   
\smallskip

\noindent  Since we have shown that $U=0$ on  $(\varOmega\cup N)\times (T,\,\infty)$, it follows that, 
for any compact subset $K$ of $(\varOmega\cup N)\times (T,\,\infty)$, 
\[
\lim_{\varepsilon\to 0+}u^\varepsilon_+=0 \  \text{ uniformly on } \ K. 
\]

\noindent To complete the proof, let $T>0$ be as above, fix any $\tau_0>0$, 
choose $\varepsilon_0>0$ so that \ $T+\tau_0<\tau(\varepsilon)$ \ for all $\varepsilon\in(0,\,\varepsilon_0)$, 
set,  for $(x,t)\in\overbar\varOmega\times [0,\,T+\tau_0] \  \text{and} \ \varepsilon\in(0,\,\varepsilon_0)$, 
\[
v^\varepsilon(x,t):=\sup\{u^\varepsilon(x,t+s)\mid 0\leq s<\tau(\varepsilon)-T-\tau_0\} 
\]
and
\[
\begin{aligned}
U(x,t):=\lim_{\lambda\to 0+}\sup\{v^\varepsilon(y,s)_+\mid &\,(y,s)\in\overbar\varOmega\times[0,\,T_0+\tau_0],
\\ 
&|y-x|+|s-t|<\lambda,\, 0<\varepsilon<\lambda\},
\end{aligned}
\]
and note that, for any $0<\varepsilon<\varepsilon_0$, $v^\varepsilon$ is a sub-solution to
$v_t^\varepsilon=L_\varepsilon v^\varepsilon$ in $\varOmega\times(0,\,T+\tau_0)$. 

\noindent It follows, as above, that 
$\,U=0\,$  in  $\,\varOmega\times[T,\,T+\tau_0)\,$
and $\,U=0\,$ in  $\,N\times(T,\,T+\tau_0)$. 
\smallskip

\noindent Let $K$ be a compact subset of $\varOmega\cup N$. Then $K\times[T+\tau_0/3,\,T+2\tau_0/3]$ is 
a compact subset of $N\times(T,\,T+\tau_0)$ and, thus,
\[
\lim_{\varepsilon\to 0+}v^\varepsilon_+=0 \ \ \ \text{ uniformly on }\ K\times[T+\tau_0/3,\,T+2\tau_0/3], 
\]
which yields 
\[
\lim_{\varepsilon\to 0+}u^\varepsilon_+=0 \ \ \ \text{ uniformly on }\ K\times[T+\tau_0/3,\,\tau(\varepsilon)-\tau_0/3]. 
\]
The proof is now complete. 
\end{proof} 

\noindent We close the section with the following generalization of Theorem \ref{thm:sp1}. 

\begin{thm} \label{thm:ac3} Under the hypotheses of Theorem \ref{thm:sp1}, if 
$K$ is a compact subset of $\varOmega\cup\argmin(V|\partial\varOmega)$, then   
$\lim_{\varepsilon\to 0+}v^\varepsilon=0 \ \text{ uniformly on } \ K.$
\end{thm}

\begin{proof} Fix  a compact $K\subset \varOmega\cup\argmin(V|\partial\varOmega)$ and $\delta>0$.  
Theorem \ref{thm:ac2} applied to 
$u^\varepsilon(x,t):=v^\varepsilon(x)-\delta$ with $N=\{y\in\partial\varOmega : g(y)<\delta\}$, gives
$\lim_{\varepsilon\to 0+}(v^\varepsilon-\delta)_+=0 \ \text{ uniformly on }\ K$,
and, hence, 
$
\lim_{\varepsilon\to 0+}v^\varepsilon_+=0  \ \text{ uniformly on } \ K. 
$
Similarly,\\ 
$\lim_{\varepsilon\to 0+}v^\varepsilon_-=0  \ \text{ uniformly on } \ K.$ \qedhere 

\end{proof}


\section{The proof of the main theorem} \label{sec:pmt}

\noindent We are now ready to prove the main theorem. Note that, in view of the linearity of pde \eqref{eq:int1}, 
it is enough to show that the following holds.

\begin{thm} \label{thm:pmt1}
Let $u^\varepsilon\in C(\overbar Q)\cap C^{2,1}(Q)$ be the solution of 
\eqref{eq:int1}, \eqref{eq:boundary} and  fix $\delta>0$. 
\begin{enumerate}
\item[(i)] There exists $T=T(\delta,g)>0$ such that, for any 
$\lambda\in(0,\,m_0)$ and any compact subset $K$ of $\varOmega$, 
\begin{equation}\label{eq:pmt1}
\lim_{\varepsilon\to 0+}\left(u^\varepsilon-g(0)-\delta\right)_+=0 
\  \text{ uniformly on }\  K\times[T,\,\e^{\lambda/\varepsilon}].   
\end{equation}
\item[(ii)] Assume that $g=g(0)$ on $\argmin(V|\partial\varOmega)$. 
There exists $T=T(\delta,g)>0$ such that, for any compact subset  
$K$ of $ \varOmega\,\cup\,\argmin(V|\partial\varOmega)$,
\begin{equation}\label{eq:pmt2}
\lim_{\varepsilon\to 0+}\left(u^\varepsilon-g(0)-\delta\right)_{+}=0 
\ \text{ uniformly on } K\times[T,\,\infty). 
\end{equation}
\item[(iii)]Assume that $g=g_0$ on  $\argmin(V|\partial\varOmega)$ for some 
$g_0\in\R$. Then, for any $\lambda\in(m_0,\,\infty)$ and any compact  subset $K$  of $\varOmega\,\cup\,\argmin(V|\partial\varOmega)$, 
\begin{equation}\label{eq:pmt3}
\lim_{\varepsilon\to 0+}\left(u^\varepsilon-g_0-\delta\right)_+=0 
\ \text{ uniformly on } \ K\times[\e^{\lambda/\varepsilon},\,\infty). 
\end{equation}
\end{enumerate}
\end{thm}


\begin{proof}[Proof of Theorem \ref{thm:pmt1}] We begin with  (i). 
Fix any $\lambda\in(0,\,m_0)$ and $\delta>0$, 
recall that $V(0)=0$ and $V>0$ in $\overbar \varOmega\setminus\{0\}$ 
and 
choose a $\gamma=\gamma(\delta,g)>0$ so that
\[
\{V\leq\gamma\}\subset\{g-g(0)-\delta\leq 0\}.  
\]
Theorem \ref{thm:sts1} yields some $r=r(\gamma)>0$ such that  
\[
\lim_{\varepsilon\to 0+}(u^\varepsilon-g(0)-\delta)_+=0 
\ \ \ \text{ uniformly on } \ B_r\times [0,\,\e^{\gamma/\varepsilon}]. 
\]
Next we use Theorem \ref{thm:ac2} to select a constant $T=T(r)>0$ such that,
for any compact subset $K$ of $\varOmega$ and $\tau_0>0$, 
\[
\lim_{\varepsilon\to 0+}(u^\varepsilon-g(0)-\delta)_+=0 \ \ \ \text{ uniformly on } \ 
K\times[T,\e^{\gamma/\varepsilon}-\tau_0].
\]
Fix  $\mu\in (\lambda,\,m_0)$. The above convergence, for \ $K=\{V \leq \mu\}$,
ensures that there exists $\varepsilon_0>0$ such that, for all $\varepsilon\in(0,\,\varepsilon_0)$,
\[
u^\varepsilon(\cdot ,T)-g(0)-2\delta\leq 0  \ \text{ on } \  \{V \leq \mu\},
\]
that is,   for all $\varepsilon\in(0,\,\varepsilon_0)$,
\[
\{V\leq \mu\} \subset \{u^\varepsilon(\cdot\,,T)-g(0)-2\delta\leq 0\}. 
\]
\begin{center}
\includegraphics[height=4cm]{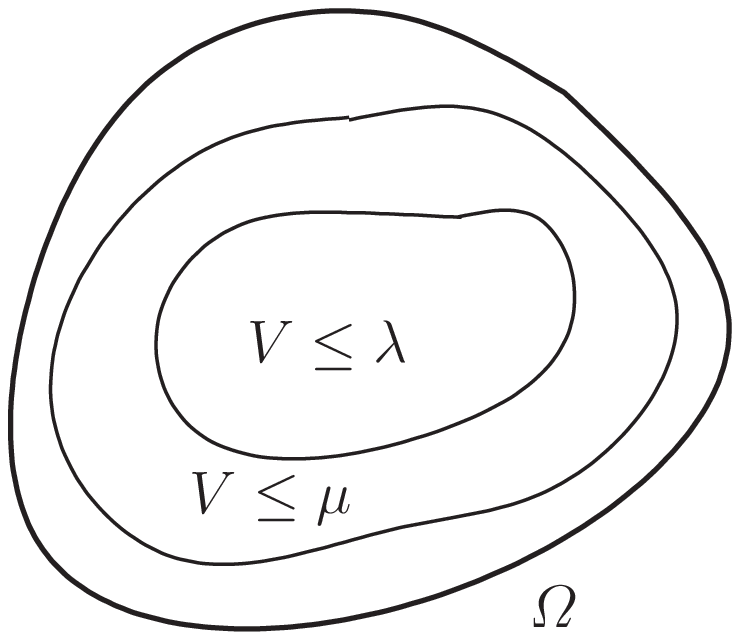}
\end{center}
Noting that $\{V\leq \mu\}$ is a neighborhood of $\{V\leq\lambda\}$, we may select $G\in C(\overbar\varOmega)$ 
so that 
\[
G=0 \  \text{ in  } \ \{V\leq\lambda\},  \ 
G\geq 0 \ \text{ on } \ \overbar\varOmega, \ \text{and } \ 
\max_{\overbar\varOmega} g-g(0)-2\delta\leq G  \text{ in  } \ \{V>\mu\}.
\]
Observe that, for all  $\varepsilon\in(0,\varepsilon_0)$,
\[
u^\varepsilon(\cdot,T)-g(0)-2\delta\leq G \ \text{ on } \ \overbar \varOmega
\ \ \ \text{and} \ \ \ 
\{V\leq\lambda\}\subset \{G\leq 0\}. 
\]
Let $U^\varepsilon\in C(\overbar Q)\cap C^{2,1}(Q)$ be the solution of \eqref{eq:int1} with 
initial-boundary  condition $U^\varepsilon=G \ \text{ on} \ \partial_{\mathrm{p}} Q$. The maximum principle implies that,  for all $(x,t)\in\overbar Q$, \\
\ $u^\varepsilon(x,T+t)-g(0)-2\delta\leq U^\varepsilon(x,t)$, while, 
in view of Theorem \ref{thm:sts1} and Theorem \ref{thm:ac2}, there exist
$r_1=r_1(\delta, G)>0$ and $T_1=T_1(r_1)>0$  respectively such that, for any compact $K\subset\varOmega$, 
\[
\lim_{\varepsilon\to 0+}(U^\varepsilon-\delta)_+=0  \ \text{ uniformly on }\ B_{r_1}\times[0,\,\e^{\lambda/\varepsilon}],
\] 
and 
\[
\lim_{\varepsilon\to 0+}(U^\varepsilon-\delta)_+=0 \ \ \ \text{ uniformly on }\ K\times[T_1,\,\e^{\lambda/\varepsilon}-T],
\]
and, hence,
\[
\lim_{\varepsilon\to 0+}(u^\varepsilon-g(0)-3\delta)_+=0 \ \ \ \text{ uniformly on } \ K\times[T+T_1,\,\e^{\lambda/\varepsilon}]. 
\]
This completes the proof of (i). 
\smallskip

\noindent Next, we prove (iii).
Fix any $\lambda>m_0$ and $\delta>0$ and let  
 $v^\varepsilon\in \Lip(\overbar\varOmega)$ be the solution of
\[
L_\varepsilon v^\varepsilon=0 \ \text{ in } \ \varOmega \ \ \ \text{and} \ \ \ 
v^\varepsilon=g \  \text{ on } \ \partial\varOmega.
\]
Theorem \ref{thm:sp2} yields $r=r(\delta)>0$ such that
\[
\lim_{\varepsilon\to 0+} (v^\varepsilon-g_0-\delta)_+=0 \ \ \ \text{ uniformly on }\ B_r.
\]
Set $N=\{x\in\partial\varOmega\mid g(x)<g_0+\delta\}$ and note that 
\[
\lim_{\varepsilon\to 0+} (v^\varepsilon-g_0-\delta)_+=0 \ \ \ \text{ uniformly on }\ B_r\cup N.
\]
Hence, by Theorem \ref{thm:ac2}, for any compact subset $K$  of $ \varOmega\cup N$,
\[
\lim_{\varepsilon\to 0+}(v^\varepsilon-g_0-\delta)_+=0 \ \  \ \text{ uniformly on } \ K.
\]
Now consider  
\[
w^\varepsilon(x,t):=u^\varepsilon(x,t)-v^\varepsilon(x) \ \ \ \text{ for } \ (x,t)\in \overbar Q.
\]
and note that 
\[
w^\varepsilon=0 \  \text{ on  }  \ \partial\varOmega\times[0,\,\infty).
\]
Then Theorem \ref{thm:lts1} yields  
\[
\lim_{\varepsilon\to 0+}w^\varepsilon=0 \ \ \ \text{ uniformly on }\ \overbar\varOmega\times[\e^{\lambda/\varepsilon},\,\infty). 
\]
Thus, we conclude that, for any compact subset $K$ of $ \varOmega\cup N$,  
\[
\lim_{\varepsilon\to 0+}(u^\varepsilon-g_0-\delta)_+=0 \ \ \ \text{ uniformly on }\ 
K\times[\e^{\lambda/\varepsilon},\,\infty),
\]
which completes the proof of (iii). 
\smallskip

\noindent To prove (ii) fix any $\delta>0$, and,   as in the proof of (i), choose 
 $r=r(\delta,g)>0$ and $\gamma>0$ such that  
\[
\lim_{\varepsilon\to 0+}(u^\varepsilon-g(0)-\delta)_+=0 
\ \ \ \text{ uniformly on } \ B_r\times [0,\,\e^{\gamma/\varepsilon}]. 
\]
Moreover, as in the proof of (iii), we 
set $N=\{x\in\partial\varOmega: g(x)-g(0)-\delta<0\}$ and use Theorem \ref{thm:ac2} to find 
$T=T(\delta,g)>0$ such that, 
for any compact subset $K$ of $\varOmega\cup N$,
\[
\lim_{\varepsilon\to 0+}(u^\varepsilon(\cdot,T)-g(0)-\delta)_+=0 \ \  \ \text{ uniformly on }\ K. 
\] 
We choose now $\lambda>m_0$ such that 
\ $\{V\leq\lambda\} \subset \varOmega\cup N$. 
It follows that there exists $\varepsilon_0>0$ such  that, if $\varepsilon\in(0,\,\varepsilon_0)$, then  
\[
u^\varepsilon(\cdot,T)-g(0)-2\delta\leq 0 \ \text{ on  } \ \{V\leq\lambda\}. 
\] 
Fix a $\mu\in(m_0,\,\lambda)$ and select  $G\in C(\overbar\varOmega)$ as in proof of (iii) 
so that 
\[
G=0 \ \text{ in } \ \{V\leq\mu\}, \ 
G\geq 0 \ \text{ on } \ \overbar\varOmega, \ \text{and} \
G\geq  \max_{\overbar\varOmega}(u^\varepsilon(\cdot,T)-g(0)-2\delta) \ \text{ in  } \ \{V>\lambda\}.
\]
Let $U^\varepsilon\in C(\overbar Q)\cap C^{2,1}(Q)$ be the solution of 
\eqref{eq:int1} with $U^\varepsilon=G$ on $\partial_{\mathrm{p}} Q$. It follows from the maximum principle that  
\ $u^\varepsilon(x,t+T)-g(0)-2\delta\leq U^\varepsilon(x,t)$ \ for all $(x,t)\in\overbar Q$, 
and $\{V\leq\mu\}\subset \{G\leq 0\}$. Combining Theorem \ref{thm:sts1} and Theorem \ref{thm:ac2},
as in the proof of (i), we deduce that 
there exists $T_1=T_1(\delta,G)>0$ such that,  for any 
compact subset $K$ of $ \varOmega\cup N$,  
\[
\lim_{\varepsilon\to 0+}(U^\varepsilon-\delta)_+=0 \ \ \ \text{ uniformly on }\ K\times[T_1,\,\e^{\mu/\varepsilon}-T].
\] 
Note that $\argmin(V|\partial\varOmega)\subset \{V\leq \mu\}$ and, hence, 
$G=0=G(0)$ on $\argmin(V|\partial\varOmega)$. Using assertion (iii), 
we see that, for any compact subset $K$ of $\varOmega\,\cup\,\argmin(V|\partial\varOmega)$,
\[
\lim_{\varepsilon\to 0+}(U^\varepsilon-\delta)_+=0 \ \ \ \text{ uniformly on } \ K\times[\e^{\mu/\varepsilon},\,\infty). 
\]
Combining these two observations, we conclude that, for any 
compact subset $K$ of $ \varOmega\cup \argmin(V|\partial\varOmega)$,  
\[
\lim_{\varepsilon\to 0+}(u^\varepsilon-g(0)-3\delta)_+=0 \ \ \ \text{ uniformly on }\ K\times[T+T_1,\,\infty). 
\] 
The proof is now complete. 
\end{proof}


\section{A semilinear parabolic equation}\label{sec:spe}

\noindent For  $f\in C(\overbar\varOmega\times\R\times\R^n)$  
we consider here the semilinear parabolic equation 
\begin{equation} \label{eq:spe1}
u_t^\varepsilon=L_\varepsilon u^\varepsilon+f_\varepsilon(x,u^\varepsilon,Du^\varepsilon)
\ \text{ in }Q.
\end{equation}

\noindent Throughout this section, in addition to (A1)--(A5), 
we make the following hypothesis. 
\begin{itemize}
\item[(A6)] %
For each $\varepsilon>0$ there exists $M(\varepsilon)>0$ 
such that, for all  
$(x,u,p)\in \overbar\varOmega\times\R\times\R^n$,  
$$
|f_\varepsilon(x,u,p)|\leq M(\varepsilon)|p|,  \  \lim_{\varepsilon\to 0+}M(\varepsilon)=0, \ \text{ and} \  u\mapsto f_\varepsilon(x,u,p) \  \text{ is nonincreasing. }
$$
%
\end{itemize} 

\noindent Note that it is immediate from (A6) that,  for all $(x,u) \in  \overbar\varOmega\times\R$, $f_\varepsilon(x,u,0)=0$.
\smallskip

\noindent In what follows, for $\phi\in C^2(\varOmega)$, we set
\[ \L_\varepsilon\phi :=L_\varepsilon\phi+f_\varepsilon(x,\phi,D\phi) \quad  \text{and} \quad    
\L_\varepsilon^+\phi:=L_\varepsilon\phi+M(\varepsilon)|D\phi|,     \]
and we remark that any sub-solution $u^\varepsilon$ of \eqref{eq:spe1} is 
also a sub-solution of $\,u^\varepsilon_t=\L_\varepsilon^+ u^\varepsilon\,$ in $\,Q$. 
\smallskip

\noindent It is possible to deal with \eqref{eq:spe1} with the nonlinear term 
$f_\varepsilon$ which depends further on the second derivatives in $x$ 
of $u^\varepsilon$, but, to make the presentation simple and to 
avoid technicalities, we restrict ourselves here to study the semilinear pde 
\eqref{eq:spe1}.  
 

\begin{thm} \label{thm:spe1}
Assume (A1)-(A6). The assertions of Theorem \ref{thm:mt1} hold for the solution
$u^\varepsilon\in C(\overbar Q)\cap C^{2,1}(Q)$ of \eqref{eq:spe1} satisfying, for $g\in C(\overbar\varOmega)$,  the initial-boundary value condition \eqref{eq:boundary}.
\end{thm}

\noindent It is not clear to the authors whether the initial-boundary value problem \eqref{eq:spe1}, \eqref{eq:boundary} has a classical solution in $C(\overbar Q)\cap C^{2,1}(Q)$. 
It is, hence,  worthwhile stating an existence and uniqueness result for viscosity solutions  
of \eqref{eq:spe1}, \eqref{eq:boundary}. 
For this  we may replace (A1) by the weaker assumption: 
\begin{itemize}
\item[(A1w)] $a$ is H\"older continuous on $\overbar\varOmega$ with exponent $\gamma>1/2$ and $b$ is continuous  
on $\overbar\varOmega$. 
\end{itemize}

\noindent We have:

\begin{thm} \label{thm:spe2} Under assumptions (A1w), (A2), (A3) and (A6) 
there exists a unique viscosity solution $u^\varepsilon\in C(\overbar Q)$ of \eqref{eq:spe1}, 
\eqref{eq:boundary}. 
\end{thm}

\noindent  We present the proof of Theorem \ref {thm:spe2}, which is rather long and technical, in the Appendix. Here we continue with 
Theorem \ref{thm:spe1}, which actually holds also for viscosity solutions of \eqref{eq:spe1}, 
\eqref{eq:boundary}.  Indeed we have:

\begin{thm} \label{thm:spe3} Assume (A1)--(A6) and $g\in C(\overbar\varOmega)$. 
The assertions of Theorem \ref{thm:spe1} hold for the (viscosity) solution
$u^\varepsilon\in C(\overbar Q)$ of \eqref{eq:spe1}, \eqref{eq:boundary}. 
\end{thm}


\noindent In view of the facts that, for any $\varepsilon>0$,  
$-f_\varepsilon(x,-u,-p)$ satisfies condition (A6) if $f_\varepsilon$ does and, if $u^\varepsilon\in C(Q)$ is a solution of \eqref{eq:spe1} then 
$v^\varepsilon:=-u^\varepsilon$ is a  solution of 
\[
v_t^\varepsilon=L_\varepsilon v^\varepsilon-f_\varepsilon(x,-v^\varepsilon,-D_xv^\varepsilon)\quad\text{ in }\ Q,
\]
Theorem \ref{thm:spe2} is an easy consequence 
of the following version of 
Theorem \ref{thm:pmt1}.

\begin{thm} \label{thm:spe4}
Assume (A1)--(A6) and $g \in C(\overbar \varOmega)$. For each $\varepsilon>0$, let $u^\varepsilon\in C(\overbar Q)$ be a sub-solution of 
\eqref{eq:spe1}, \eqref{eq:boundary}. Fix $\delta>0$. 
\begin{enumerate}
\item[(i)] There exists $T=T(\delta,g)>0$ such that, for any 
$\lambda\in(0,\,m_0)$ and any compact subset $K$ of $ \varOmega$, 
\[
\lim_{\varepsilon\to 0+}\left(u^\varepsilon-g(0)-\delta\right)_+=0 
\quad \text{ uniformly on }\  K\times[T,\,\e^{\lambda/\varepsilon}].   
\]
\item[(ii)] Assume that $g=g(0)$ on $\argmin(V|\partial\varOmega)$. 
There exists a constant $T=T(\delta,g)>0$ such that,  for any  
compact subset  $K$ of $\varOmega\,\cup\,\argmin(V|\partial\varOmega)$,
\[
\lim_{\varepsilon\to 0+}\left(u^\varepsilon-g(0)-\delta\right)_{+}=0 
\quad \text{ uniformly on } K\times[T,\,\infty). 
\]
\item[(iii)]Assume that $g=g_0$ on $\argmin(V|\partial\varOmega)$ for some 
$g_0\in\R$. Then, for any $\lambda\in(m_0,\,\infty)$ and any compact subset $K$ of $\varOmega\,\cup\,\argmin(V|\partial\varOmega)$, 
\[
\lim_{\varepsilon\to 0+}\left(u^\varepsilon-g_0-\delta\right)_+=0 
\quad\text{ uniformly on } \ K\times[\e^{\lambda/\varepsilon},\,\infty). 
\]
\end{enumerate}
\end{thm}


\noindent The proof of Theorem \ref{thm:spe4} parallels  that of Theorem \ref{thm:pmt1}. 
Instead of giving the detailed proof, we indicate here 
its major differences from that of Theorem \ref{thm:pmt1}.  
\smallskip


\noindent  Choose $r$, 
$W_r$, $\mu$, $\eta$ and $C$ as those at the beginning of Section \ref{sec:sts}, let $v^\varepsilon\in C^2(\overbar\varOmega)$ be the function defined by \eqref{eq:sts1a}   and observe that 
\[
f_\varepsilon(x,v^\varepsilon,D_xv^\varepsilon)
\leq M(\varepsilon)|D_xv^\varepsilon|
\leq \frac{v^\varepsilon}{\varepsilon} M(\varepsilon)|DW_r|
\]
and 
\[
\L_\varepsilon v^\varepsilon
\leq \frac{v^\varepsilon}{\varepsilon}\left(H(x,DW_r)+\varepsilon C+M(\varepsilon)\|DW_r\|_{\infty,\varOmega}
\right). 
\]
Select $\varepsilon_0>0$ so that, for all $\varepsilon\in(0,\,\varepsilon_0)$,
\[
\varepsilon C+M(\varepsilon)\|DW_r\|_{\infty,\varOmega}\leq 1\wedge \eta,
\]
and observe that, for any $\varepsilon\in(0,\,\varepsilon_0)$,
\[
\L_\varepsilon v^\varepsilon(x)
\leq
\begin{cases}
0&\text{ in } \ \varOmega\setminus B_r,\\[3pt]\displaystyle 
\frac{2v^\varepsilon}{\varepsilon}&\text{ in } \ B_r. 
\end{cases}
\]
The function $w^\varepsilon\in C^2(\overbar Q)$ defined by  
$\,w^\varepsilon(x,t):=v^\varepsilon(x)+R_\varepsilon t$, 
with $\,R_\varepsilon:=(2/\varepsilon)\|v^\varepsilon\|_{\infty,B_r}\,$ satisfies
$\,w_t^\varepsilon\geq \L_\varepsilon w^\varepsilon\,$  in  $\, Q.$
\smallskip

\noindent The next assertion (Theorem \ref{thm:spe5}) is similar to Theorem \ref{thm:sts1}. 
Its proof follows by  a straightforward adaptation of the proof  of Theorem 
\ref{thm:sts1} with the above choice of function $w^\varepsilon$. 

\begin{thm}\label{thm:spe5} For each $\varepsilon>0$ 
let $u^\varepsilon\in\USC(\overbar Q)$ be a sub-solution of \eqref{eq:spe1},  
\eqref{eq:boundary} and let 
$\lambda>0$ be such that \ $\{V\leq\lambda\}\subset\{g\leq 0\}$. 
For any $\delta>0$ there exists $r>0$ such that 
\[
\lim_{\varepsilon\to 0+}(u^\varepsilon-\delta)_+=0 \quad\text{ uniformly on } \ 
B_r\times [0,\,\e^{\lambda/\varepsilon}].
\]
\end{thm}

\noindent Theorem \ref{thm:lts1} can be reformulated for sub-solutions of \eqref{eq:spe1} as follows. 

\begin{thm}\label{thm:spe6} Fix $\lambda>m_0$ and, for each $\varepsilon>0$, 
let $u^\varepsilon\in\USC(\overbar  Q)$ be a sub-solution 
of \eqref{eq:spe1}. Assume that
$u^\varepsilon\leq 0 \ \text{ on } \ \partial\varOmega\times [0,\,\infty) \ \text{ and } \sup_{\varepsilon>0} \|u^\varepsilon\|_{\infty, \overbar Q}<\infty.$
Then 
\[
\lim_{\varepsilon\to 0+}u^\varepsilon_+=0 \quad\text{ uniformly on }\ \overbar\varOmega\times[\e^{\lambda/\varepsilon},\,\infty). 
\]
\end{thm}

\noindent Let $W$, $\eta$, $\delta$, $\mu$ and $v^\varepsilon$ be as in the proof Theorem \ref{thm:lts1}. 
We deduce, following the arguments in the proof of Theorem \ref{thm:lts1}, that, in the sub-solution sense,
\[
\L_\varepsilon v^\varepsilon
\geq \frac{v^\varepsilon}{\varepsilon}\Big(
-\frac{\varepsilon}{\eta}
+\eta
-M(\varepsilon)\|DW\|_{\infty,\,\varOmega}\Big) \  \text{ in } \varOmega.
\]
Fix $\varepsilon_0>0$ so that, for all $\varepsilon\in(0,\,\varepsilon_0)$, 
\[
\frac {\varepsilon}\eta+M(\varepsilon)\|DW\|_{\infty,\,\varOmega}<\frac \eta 2, 
\ \ \ 
\text{ and, hence, }
\ \ \ \L_\varepsilon v^\varepsilon\geq\frac{\eta v^\varepsilon}{2\varepsilon} \quad\text{ in }\varOmega. 
\] 
Define $w^\varepsilon\in\Lip(\overbar Q)$ as in the proof of Theorem \ref{thm:lts1}, that is, for $\gamma\in(0,\,\eta]$,  set  
\[
w^\varepsilon(x,t):=1+\e^{-\delta/\varepsilon}-v^\varepsilon(x)-\frac{\gamma}{2\varepsilon}\e^{-\mu/\varepsilon}\,t,
\]
and then 
follow the proof of Theorem \ref{thm:lts1} with 
$w^\varepsilon$ above, 
to conclude the proof of Theorem \ref{thm:spe6}. 
\smallskip

\noindent A review of the proof of Theorem \ref{thm:ac2} shows 
that, with a minor modification of the function $\psi$, 
the assertion of Theorem 
\ref{thm:ac2} holds true  for sub-solutions $u^\varepsilon\in \USC(\overbar Q)$ of \eqref{eq:spe1}. To prove the first claim of Theorem \ref{thm:spe4}, we just need to follow the proof of part (i) of Theorem \ref{thm:pmt1}, with Theorem 
\ref{thm:sts1} replaced by Theorem \ref{thm:spe5} 
and with Theorem \ref{thm:ac2} replaced by the corresponding assertion 
for sub-solutions $u^\varepsilon\in \USC(\overbar Q)$ of \eqref{eq:spe1}. 
\smallskip

\noindent Now we discuss a version of Theorem \ref{thm:sp2} for sub-solutions of 
\begin{equation}\label{eq:spe4}
\begin{cases}
\L_\varepsilon v^\varepsilon=0&\text{ in }\ \varOmega,\\[3pt]
v^\varepsilon=g &\text{ on }\ \partial\varOmega,
\end{cases}
\end{equation} 
with  $g\in C(\overbar\varOmega)$. The existence and uniqueness of a solution in $C(\overbar \varOmega)$ 
of \eqref{eq:spe4} follow similarly  to the case of 
Theorem \ref{thm:spe1}. 
\smallskip

\noindent Following the proof of Theorem \ref{thm:sp1} we obtain: 

\begin{thm}
\label{thm:spe7} For each $\varepsilon>0$ let $v^\varepsilon\in\USC(\overbar Q)$
 be a sub-solution of \eqref{eq:spe4}. Assume that 
 $g\leq g_0$  on $\argmin(V|\partial\varOmega)$ for some constant $g_0$. Then, for any $\delta>0$, 
 there exists $r>0$ such that 
 \[
 \lim_{\varepsilon\to 0+}(v^\varepsilon-g_0-\delta)_+=0 \ \ \ \text{ uniformly on }\ B_r.
 \]
 \end{thm}

\noindent When following the proof of Theorem \ref{thm:sp1}, 
one needs to replace $G$ and $u^\varepsilon$, respectively, by 
the function $G(x)=g(x)-g_0-\delta$ and the solution $u^\varepsilon\in C(\overbar Q)$ of 
\begin{equation}\label{eq:spe5}
u_t^\varepsilon=\L_\varepsilon^+ u^\varepsilon\ \text{ in }\ Q  
\ \text{ with} \ u^\varepsilon=h  \ \text{ on }\ \partial_{\mathrm{p}} Q,
\end{equation}
where $h\in C(\overbar \varOmega)$ is chosen as in the proof of Theorem \ref{thm:sp1} with the present 
choice of $G$. Once it is shown that $w^\varepsilon:=v^\varepsilon-u^\varepsilon$ is a sub-solution of 
\eqref{eq:spe5}, the rest of the argument goes exactly as in the proof 
of Theorem \ref{thm:sp1}.  Thus, the following lemma completes the proof of Theorem \ref{thm:spe7}.

\begin{lem}\label{lem:spe2} 
For a given $\varepsilon>0$ let $v\in\USC(Q)$ and $u\in\LSC(Q)$ be respectively a sub-solution and a 
super-solution of \eqref{eq:spe5}. Then  $w:=v-u$ is a
sub-solution of \eqref{eq:spe5}. 
\end{lem}

\begin{proof} Let $\phi\in C^2(Q)$ and $(\hat x,\hat t)\in Q$ be such that $w-\phi$ achieves 
a strict maximum at $(\hat x,\hat t)$. We need to show that  
$\,\phi_t\leq \L_\varepsilon^+\phi\,$ at $\,(\hat x,\hat t)$. 
\smallskip

\noindent We argue by contradiction and thus assume that this inequality does not hold. 
In this case we may choose $r>0$ so that $Q_r:=B_r(\hat x)\times(\hat t-r, \hat t+r) \subset Q$ and 
\[
\phi_t>\L_\varepsilon^+\phi  \ \text{ in  } \ Q_r:=B_r(\hat x)\times(\hat t-r,\,\hat t+r). 
\]

\noindent It is easily seen that $v-\phi$ is a sub-solution of \eqref{eq:spe5} in $Q_r$. Moreover, 
there is a  comparison between $v-\phi$ and $u$ (see the comparison 
principle at the beginning of the proof of Theorem \ref{thm:spe2} below), 
that is, we have
\[
\max_{\overbar{Q}_r}(v-u-\phi)\leq \max_{\partial_{\mathrm{p}} Q_r}(v-u-\phi), 
\] 
which is a contradiction since $w-\phi=v-u-\phi$ has a strict maximum at $(\hat x,\hat t)\in Q_r$. 
\end{proof} 
\smallskip

\noindent The proof of part (iii), (ii) of Theorem \ref{thm:spe4} follows as that of part (iii), (ii) 
of Theorem \ref{thm:pmt1} 
once $v^\varepsilon$ is chosen as the solution of 
\[
\L_\varepsilon^+ v^\varepsilon=0  \ \text{ in } \ \varOmega \ \ \text{ with } \ \ 
v^\varepsilon=g(x) \ \text{ on } \ \partial\varOmega,
\]
and Theorems \ref{thm:sp2}, \ref{thm:ac2}, \ref{thm:lts1} and 
\ref{thm:sts1} are replaced by those for sub-solutions of \eqref{eq:spe1} and \eqref{eq:spe4}.

\section{ Appendix: The well posedness of the semilinear problem}

\noindent 
We need the following lemma. Its proof is postponed for later.

\begin{lem}\label{lem:spe1}There exists a constant $\lambda_0>0$ such that, 
for all $\, y\in\partial\varOmega\,$  and $\,\lambda\in(0,\,\lambda_0),$ 
$y+\lambda \nu(y)\in\R^n\setminus \overbar\varOmega$. 
Moreover, if 
$\delta(\lambda):=\min_{y\in\partial\varOmega}\dist(y+\lambda\nu(y),\varOmega)$,  then $\lim_{\lambda\to 0+}  \delta(\lambda)/\lambda=1.$
\end{lem}


\begin{proof}[Outline of the proof of Theorem \ref{thm:spe2}]  The  uniqueness follows from  the following comparison principle.
If $v\in \USC(\overbar Q)$ and $w\in\LSC(\overbar Q)$ are, respectively, 
a sub-solution and a super-solution of \eqref{eq:spe1} and 
$v\leq w$ on $\partial_{\mathrm{p}} Q$, then $u\leq w$ in $Q$. The comparison above 
is a special, parabolic version of (i) of Theorem III.1 in \cite{IL90}
and can be proved in the same spirit as the latter theorem.  
A useful comment here is that the proof of (i) of Theorem III.1 in \cite{IL90} works even when 
the constant $C_R$ in 
the assumption (3.2) there replaced by $C_R(1+|p|^\gamma)$, 
for some $\gamma\in(0,\,1)$.   
\smallskip

\noindent The existence of a solution follows from Perron's method provided we construct 
appropriate sub-solution and super-solution of 
\eqref{eq:spe1} in $Q$. 
\smallskip

\noindent To this end, let $\lambda_0 \in (0,1)$ and $\delta :(0,\,\lambda_0)\to (0,\,\lambda_0)$ be as in Lemma \ref{lem:spe1}. 
For each  $y\in\partial\varOmega$ and $\lambda\in(0,\,\lambda_0)$ 
set \ 
$z:=y+\lambda\nu(y)$ and, for $\alpha>0$  a constant which depends on $\lambda$ to be fixed later, define $u_{\mathrm{b}}, v_{\mathrm{b}}\in C^\infty(\overbar\varOmega)$ by 
\[
u_{\mathrm{b}}(x):=u_{\mathrm{b}}(x;y,\lambda):=\e^{-\alpha(|x-z|^2-\delta(\lambda)^2)} \quad\text{ and }\quad 
v_{\mathrm{b}}(x):=v_{\mathrm{b}}(x;y,\lambda):=1-u_{\mathrm{b}}(x),
\]
\smallskip

\noindent Observe that, if $d:=\diam(\varOmega)$, then  for all $x\in\overbar\varOmega$,  
\[
\delta(\lambda)\leq |x-z|\leq d+1.
\]
\noindent Next we estimate $\L_\varepsilon v_{\mathrm{b}}$ from above to find
\[
\begin{aligned}
\L_\varepsilon v_{\mathrm{b}}(x)&\,\leq u_{\mathrm{b}}(x)\big\{\varepsilon\left(2\alpha\tr a(x)-4\alpha^2 a(x)(x-z)\cdot (x-z)\right)+2\alpha M(\varepsilon)|x-z|
\\&\, \qquad 
+2\alpha|x-z|\|b\|_{\infty,\varOmega}\big\}
\\&\,\leq \alpha u_{\mathrm{b}}(x)\{\varepsilon\left(
2n\theta^{-1}-4\alpha\theta \delta(\lambda)^2\right)+2(d+1)(M(\varepsilon)+\|b\|_{\infty,\varOmega})\}. 
\end{aligned}
\]
Fix $\Lambda>0$ and $\alpha=\alpha(\lambda)>0$ so that 
\[
\varepsilon\left(
2n\theta^{-1}-4\theta \Lambda\right)
+2(d+1)\left(M(\varepsilon)+\|b\|_{\infty,\varOmega}\right)=0 \ \text{ and } \ \alpha\delta(\lambda)^2=\Lambda,
\] 
and note that, with this choice,  
\begin{equation}\label{eq:spe2}
\L_\varepsilon v_{\mathrm{b}}\leq 0 \  \text{ on } \  \overbar\varOmega. 
\end{equation}
We also observe that
\[
\begin{aligned}
v_{\mathrm{b}}(y)&\,=1-\exp\left(-\alpha\left(\lambda^2-\delta(\lambda)^2\right)\right)
=1-\exp\left(-\Lambda\left(\lambda^2/\delta(\lambda)^2-1\right)\right),\\ 
v_{\mathrm{b}}(x)&\,\geq 1-\exp\left(-\alpha\left(\delta(\lambda)^2-\delta(\lambda)^2\right)\right)=0 \ \ \text{ for all }\ x\in\overbar\varOmega, 
\end{aligned}
\]
and, for any $x\in\overbar\varOmega\setminus B_{3\lambda}(y)$, 
\[
\begin{aligned}
v_{\mathrm{b}}(x)&\,\geq 1-\exp\left(-\alpha\left((|x-y|-|y-z|)^2-\delta(\lambda)^2\right)\right)
\geq 1-\exp\left(-\alpha\left(4\lambda^2-\delta(\lambda)^2\right)\right)
\\&\,> 1-\exp\left(-\alpha\delta(\lambda)^2\right)=1-\e^{-\Lambda}.
\end{aligned}
\]
Lemma \ref{lem:spe1} together with the first observation above yields
\begin{equation}\label{eq:spe3}
\lim_{\lambda\to 0+}v_{\mathrm{b}}(y;y,\lambda)=0. 
\end{equation}

\noindent Next let $\omega$ denote the modulus of continuity of $g$, choose $A>0$ so that 
$A\left(1-\e^{-\Lambda}\right)>\omega(d),$
and observe that, for any $x\in\overbar\varOmega$ and $y\in\partial\varOmega$, 
\[
g(x)\leq g(y)+\omega(d)\leq g(y)+Av_{\mathrm{b}}(x;y,\lambda) \ \ \ \text{ if } \ x\not\in B_{3\lambda}(y),
\]
and 
\[
g(x)\leq g(y)+\omega(3\lambda) \ \ \ \text{ if }\ x\in B_{3\lambda}(y). 
\]
Hence, for all $x\in\overbar\varOmega$, 
\[
g(x)\leq g(y)+\omega(3\lambda)+Av_{\mathrm{b}}(x;y,\lambda). 
\]
Thus, setting, for $x\in \overbar\varOmega$,
\[
w_{\mathrm{b}}(x):=\inf\{g(y)+\omega(3\lambda)+A v_{\mathrm{b}}(x;y,\lambda): \lambda\in(0,\,\lambda_0),\, y\in\partial\varOmega\}
\]
and recalling \eqref{eq:spe2} and \eqref{eq:spe3}, we deduce that $w_{\mathrm{b}} \in \USC(\overbar\varOmega)$ is a  
super-solution of $\L_\varepsilon w_{\mathrm{b}}=0$ in $\varOmega$, $\,w_{\mathrm{b}}\geq g\,$ 
on $\,\overbar\varOmega\,$ 
and $\,w_{\mathrm{b}}=g\,$ on $\,\partial\varOmega$. 
\smallskip

\noindent Next, let $\gamma>0$, choose $B=B(\gamma)>0$ so that 
$B\gamma^2\geq\omega(d)$, for  
 $y\in\overbar\varOmega$, define $v_{\mathrm{i}}=v_{\mathrm{i}}(\cdot\,,y,\gamma)\in C^\infty(\overbar\varOmega)$
by
\[
v_{\mathrm{i}}(x):=g(y)+B|x-y|^2+\omega(\gamma), 
\]
and observe that 
\[
v_{\mathrm{i}}\geq g  \  \text{ on } \ \overbar \varOmega \  \text{ and } \  
v_{\mathrm{i}}(y)=g(y)+\omega(\gamma). 
\] 
Choose  $C(\gamma)>0$ so that 
\[ 
\L_\varepsilon v_{\mathrm{i}}\leq C(\gamma) \ \text{ in } \ \overbar \varOmega \ \text{ for all }  
v_{\mathrm{i}}=v_{\mathrm{i}}(\cdot\,;y,\gamma) \ \text{ with } \ y\in\overbar\varOmega,
\]
set, for  $(x,t)\in\overbar Q$, 
\[
w_{\mathrm{i}}(x,t):=\inf\{v_{\mathrm{i}}(x;y,\gamma)+C(\gamma)\,t \mid y\in\overbar\varOmega,\,\gamma>0\}, 
\]
and observe that $\,w_{\mathrm{i}} \in \USC(\overbar \varOmega)\,$ is a super-solution of \eqref{eq:spe1}, 
$\,g\leq w_{\mathrm{i}}\,$ on  $\,\overbar Q\,$ and 
$\,w_{\mathrm{i}}(\cdot,0)=g\,$ on $\,\overbar\varOmega$.
\smallskip

\noindent  Now,  for $(x,t)\in\overbar Q$, let
\[
w(x,t):=\min\{w_{\mathrm{b}}(x),\,w_{\mathrm{i}}(x,t)\}; 
\]
it is immediate that  $w\in\USC(\overbar Q)$ is a super-solution of \eqref{eq:spe1}, 
and, in addition, $w=g$ 
on $\partial_{\mathrm{p}} Q$ and $w \geq g $ on $\overbar Q$. 
\smallskip

\noindent Similarly, we can build a  sub-solution $z\in \LSC(\overbar Q)$ of \eqref{eq:spe1} such that 
 $z=g$ on $\partial_{\mathrm{p}} Q$
and $z\leq g$ on  $ \overbar Q$. Perron's method  together 
with the comparison claim mentioned at beginning of the ongoing proof yields 
a solution $u\in C(\overbar Q)$ of \eqref{eq:spe1} such that $z\leq u\leq w$ on $\overbar Q$. 
The last inequality implies that $u=g$ on  $\partial_{\mathrm{p}} Q$.  
\end{proof}

\noindent We present now the 
\begin{proof}[Proof of Lemma \ref{lem:spe1}] 
Let $\rho \in C^1(\R^n)$ be a defining function of $\varOmega$, that is  
\ $\varOmega=\{x\in\R^n: \rho(x)<0\}$ \ and  \ $D\rho\not=0$ if $\rho =0$; its existence is guaranteed by the assumed regularity of the 
boundary of $\varOmega$. 
\smallskip

\noindent Since,  for any $y\in\partial\varOmega$, there exists $\theta_0\in(0,\,1)$ such that
\[
\rho(y+\lambda\nu(y))=\lambda D\rho(y+\theta_0\lambda\nu(y))\cdot\nu(y)
\]
we may deduce that there exists $\lambda_0>0$ 
such that, for all  $y\in\partial\varOmega \ \text{ and} \ \lambda\in(0,\,\lambda_0)$,
\[
y+\lambda\nu(y)\in\R^n\setminus \overbar\varOmega.
\]

\noindent To show that $\lim_{\lambda\to 0+}\delta(\lambda)/\lambda=1$, we first note that $\delta(\lambda)\leq\lambda$ and  
assume by contradiction that $\liminf_{\lambda\to 0+}\delta(\lambda)/\lambda<1$.  It then follows that there exist $\delta_0\in(0,\,1)$ 
and a sequence $\{\lambda_j\}_{j \in \N} \subset (0,\,\lambda_0)$ such that, as $j \to \infty$,   $\lambda_j \to 0$ and 
$\delta(\lambda_j)/\lambda_j\leq\delta_0$ for all $j$. Moreover,
for each $j\in\N$ there are $y_j,\,\xi_j\in\partial\varOmega$ 
such that  
\[
\delta(\lambda_j)=|y_j+\lambda_j\nu(y_j)-\xi_j|. 
\]
We may assume by passing,   if needed,  to a subsequence,  that, as $j\to\infty$,  $y_j \to y_0$ 
for some $y_0\in\partial\varOmega$. It is then clear that $\xi_j\to y_0$ as $j\to\infty$. 
\smallskip

\noindent Since $\xi_j$ is a closest point of $\partial\varOmega$ to $y_j+\lambda_j\nu(y_j)$, we have 
\[
\xi_j+\delta(\lambda_j)\nu(\xi_j)=y_j+\lambda_j\nu(y_j). 
\]
Hence, noting that, for some $\theta_j,\,\tilde\theta_j\in(0,\,1)$,
\[\begin{cases}
\rho(y_j+\lambda_j\nu(y_j))=\lambda_j D\rho(y_j+\theta_j\lambda_j\nu(y_j))\cdot\nu(y_j),\\[3pt]
\rho(\xi_j+\delta(\lambda_j)\nu(\xi_j))=\delta(\lambda_j) D\rho(\xi_j+\tilde\theta_j\delta(\lambda_j)\nu(\xi_j))\cdot\nu(\xi_j),
\end{cases}
\]
we find 
\[
\lambda_j D\rho(y_j+\theta_j\lambda_j\nu(y_j))\cdot\nu(y_j)
=\delta(\lambda_j) D\rho(\xi_j+\tilde\theta_j\delta(\lambda_j)\nu(\xi_j))\cdot\nu(\xi_j),
\]
which, in the limit as $j\to\infty$, yields 
\[
\lim_{j\to\infty}\frac{\delta(\lambda_j)}{\lambda_j}
=\lim_{j\to\infty}\frac{D\rho(y_j+\theta_j\lambda_j\nu(y_j))\cdot\nu(y_j)}{D\rho(\xi_j+\tilde\theta_j\delta(\lambda_j)\nu(\xi_j))\cdot\nu(\xi_j)}
=1,
\]
a contradiction to the inequality $\,\delta(\lambda_j)/\lambda_j\leq \delta_0<1\,$ for $\,j\in\N$. 
\end{proof}

\bibliographystyle{plain}
\bibliography{metastability}

\bye